\documentclass[10pt]{amsart}
\usepackage{graphicx,amsfonts,amssymb,amsmath,amsthm}
\usepackage{geometry}
\usepackage[all]{xy}
\usepackage[colorlinks=true,urlcolor=blue]{hyperref}
\usepackage{mathrsfs}

\theoremstyle{plain} 
\newtheorem{theorem}    {Theorem}
\newtheorem{theoremletter}{Theorem}

\newtheorem{lemma}      [theorem]{Lemma}
\newtheorem{corollary}  [theorem]{Corollary}
\newtheorem{proposition}[theorem]{Proposition}

\newtheorem{conjecture} [theorem]{Conjecture}

\theoremstyle{definition}

\theoremstyle{remark}
\newtheorem{remark}    [theorem]          {Remark}

\numberwithin{equation}{section}

\newcommand{\SU}{\operatorname{SU}}
\newcommand{\diag}{\operatorname{diag}}
\newcommand{\Res}{\operatorname{Res}}
\newcommand{\cl}{\operatorname{cl}}
\newcommand{\spin}{\operatorname{spin}}
\newcommand{\SO}{\operatorname{SO}}
\newcommand{\GSpin}{\operatorname{GSpin}}
\newcommand{\Sp}{\operatorname{Sp}}

\newcommand{\AI}{\operatorname{AI}}
\newcommand{\sign}{\operatorname{sign}}

\newcommand{\Gal}{\operatorname{Gal}}

\newcommand{\rig}{\operatorname{rig}}

\newcommand{\Frss}{\operatorname{Fr-ss}}
\newcommand{\WD}{\operatorname{WD}}

\newcommand{\Iw}{\operatorname{Iw}}

\newcommand{\BC}{\operatorname{BC}}
\newcommand{\std}{\operatorname{std}}
\renewcommand{\sl}{\mathfrak{sl}}
\newcommand{\SL}{\operatorname{SL}}
\newcommand{\Ind}{\operatorname{Ind}}
\newcommand{\End}{\operatorname{End}}
\newcommand{\GSp}{\operatorname{GSp}}
\newcommand{\dR}{\operatorname{dR}}
\newcommand{\St}{\operatorname{St}}
\newcommand{\rk}{\operatorname{rk}}

\newcommand{\Sym}{\operatorname{Sym}}

\newcommand{\coker}{\operatorname{coker}}

\newcommand{\Fil}{\operatorname{Fil}}
\newcommand{\gr}{\operatorname{gr}}

\newcommand{\D}{\operatorname{D}}

\newcommand{\res}{\operatorname{res}}

\newcommand{\Hom}{\operatorname{Hom}}

\newcommand{\GL}{\operatorname{GL}}

\newcommand{\cycl}{\operatorname{cycl}}
\newcommand{\rec}{\operatorname{rec}}
\newcommand{\Tr}{\operatorname{Tr}}
\newcommand{\Frob}{\operatorname{Frob}}

\newcommand{\st}{\operatorname{st}}
\newcommand{\cris}{\operatorname{cris}}
\newcommand{\Ad}{\operatorname{Ad}}
\newcommand{\sqtimes}{\boxtimes}

\newcommand{\QQ}{\mathbb{Q}}
\newcommand{\RR}{\mathbb{R}}
\newcommand{\CC}{\mathbb{C}}
\newcommand{\ZZ}{\mathbb{Z}}
\renewcommand{\AA}{\mathbb{A}}
\newcommand{\Qp}{\QQ_p}
\newcommand{\Zp}{\ZZ_p}
\newcommand{\ol}[1]{\overline{#1}}

\renewcommand{\c}{\chi_{\cycl}}
\newcommand{\into}{\hookrightarrow}

\renewcommand{\mod}[1]{\text{ (mod }#1)}
\renewcommand{\phi}{\varphi}

\newcommand{\Drig}{\D_{\rig}^\dagger}

\hypersetup{pdfauthor={Robert Harron and Andrei Jorza},pdftitle={On symmetric power \textit{L}-invariants of Iwahori level Hilbert modular forms},pdfkeywords={\textit{p}-adic \textit{L}-functions, Iwasawa theory, \textit{L}-invariants, symmetric powers}}

\begin{document}

\title{On symmetric power $\mathcal{L}$-invariants of Iwahori level Hilbert modular forms}
\author{Robert Harron}
\address{R. Harron: University of Wisconsin, Department of
  Mathematics, 480 Lincoln Drive, Madison, WI 53706}
\email{rharron@math.wisc.edu}
\thanks{The first author is partially supported by NSA Young Investigator Grant \#H98230-13-1-0223 and NSF RTG Grant ``Number Theory and Algebraic Geometry at the University of Wisconsin''.}

\author{Andrei Jorza}
\address{A. Jorza:  University of Notre Dame, Department of
  Mathematics, 275 Hurley Hall, Notre Dame, IN 46556}
\email{ajorza@nd.edu}
\date{\today}

\subjclass[2010]{11R23, 11F80, 11R34, 11F67}

\begin{abstract}
We compute the arithmetic $\mathcal{L}$-invariants (of
Greenberg--Benois) of twists of symmetric powers of $p$-adic
Galois representations attached to Iwahori level Hilbert modular
forms (under some technical conditions). Our method uses the automorphy of symmetric powers and the study of analytic Galois representations on $p$-adic families of automorphic forms over symplectic and unitary groups. Combining these families with some explicit plethysm in the representation theory of $\GL(2)$, we construct global Galois cohomology classes with coefficients in the symmetric powers and provide formulae for the $\mathcal{L}$-invariants in terms of logarithmic derivatives of Hecke eigenvalues.
\end{abstract}

\maketitle

\tableofcontents

\section*{Introduction}
The $p$-adic interpolation of special values of $L$-functions has been critical to understanding their arithmetic, providing, for instance, the link between values of the Riemann zeta function and the class groups of cyclotomic fields. However, it can happen that the $p$-adic $L$-function of a motive $M$ vanishes at a point of interpolation despite the fact that the classical $L$-value is non-zero. To recuperate an interpolation property, one predicts that a \textit{derivative} of the $p$-adic $L$-function is related to the $L$-value and one defines the (analytic) $\mathcal{L}$-invariant of $M$ as the ratio of these quantities. The name of the game then becomes to provide an \textit{arithmetic} meaning for the $\mathcal{L}$-invariant and show that it is non-zero.

The phenomenon of $\mathcal{L}$-invariants first arose in the work of Ferrero--Greenberg \cite{FeG78} on $p$-adic Dirichlet $L$-functions. Later on, in their work \cite{mazur-tate-teitelbaum} on formulating a $p$-adic analogue of the Birch and Swinnerton-Dyer conjecture, Mazur--Tate--Teitelbaum encountered this behaviour in the $p$-adic $L$-function of an elliptic curve over $\QQ$ with split multiplicative reduction at $p$. They conjectured a formula for the $\mathcal{L}$-invariant in terms of the $p$-adic Tate parameter of $E$ (and coined the term $\mathcal{L}$-invariant). This was proved in \cite{greenberg-stevens:p-adic-L} by Greenberg and Stevens by placing the elliptic curve in a Hida family, i.e.\ a $p$-adic family of modular forms of varying weight. Their proof proceeded in two steps: first, relate the Tate parameter to the derivative in the weight direction of the $U_p$-eigenvalue of the Hida family, and then use the functional equation of the two-variable $p$-adic $L$-function of the Hida family to relate this derivative to the analytic $\mathcal{L}$-invariant. We think of the Tate parameter formula as being a conjectural \textit{arithmetic} $\mathcal{L}$-invariant and the Greenberg--Stevens method as first linking it to a derivative of a Hecke eigenvalue and then appealing to analytic properties of several-variable $p$-adic $L$-functions to connect with the analytic $\mathcal{L}$-invariant. In the case of a $p$-ordinary motive $M$, Greenberg used an in-depth study of  the ordinary filtration to conjecture an arithmetic formula for the $\mathcal{L}$-invariant of $M$ in terms of its Galois cohomology (\cite[Equation (23)]{G94}). In \cite{benois:L-invariant}, Benois generalized this to the non-ordinary setting by passing to the category of $(\varphi,\Gamma)$-modules over the Robba ring and using triangulations. This article aims to establish the first step of the Greenberg--Stevens method for symmetric powers of Hilbert modular forms with respect to the Greenberg--Benois arithmetic $\mathcal{L}$-invariant; the second step being of a different nature (and the corresponding $p$-adic $L$-functions not known to exist!).

Before presenting the main results of this article, we introduce the trivial zero conjecture as formulated by Benois in
\cite{benois:L-invariant}. Let $p$ be an odd prime and let $\rho:G_{\QQ}\to \GL(n, \overline{\QQ}_p)$ be a continuous
Galois representation which is unramified at all but finitely
many primes, and semistable at $p$. One
has an $L$-function $L(\rho,s)=\prod_\ell L_\ell(\rho|_{G_{F_\ell}}, s)$
where $L_\ell(\rho|_{G_{\QQ_\ell}},s) = \det(1-\Frob_\ell
\ell^{-s}|\rho^{I_{\ell}})^{-1}$ for $\ell\neq p$ (here $\Frob_\ell$ is
the geometric Frobenius) while $L_p(\rho|_{G_{\QQ_p}}, s) = \det(1-\phi
p^{-s}|\D_{\cris}(\rho|_{G_{\QQ_p}}))^{-1}$. It is conjectured that $L(V,s)$ has meromorphic continuation to $\CC$ and 
that there exist Gamma factors $\Gamma(\rho,s)$ and
$\Gamma(\rho^*(1),s)$ such that
$\Gamma(\rho,s)L(\rho,s)=\varepsilon(\rho,s)\Gamma(\rho^*(1),-s)L(\rho^*(1),-s)$ for an epsilon factor
$\varepsilon(\rho,s)$ of the form $A\cdot B^s$. If $D\subset \D_{\st}(\rho|_{G_{\QQ_p}})$
is a regular submodule (see \S\ref{sect:regular submodules general}), it is
expected that there exists an analytic $p$-adic $L$-function
$L_p(\rho, D, s)$ such that
\[L_p(\rho,D,0)=\mathcal{E}(\rho,D)\frac{L(\rho,s)}{\Omega_\infty(\rho)}\]where
$\mathcal{E}(\rho,D)$ is a product of Euler-like factors and
$\Omega_\infty(\rho)$ is a transcendental period.
\begin{conjecture}[{\cite[Trivial Zero
  Conjecture]{benois:L-invariant}}]\label{conjecture:zero}If $\rho$ is critical,
$L(\rho,0)\neq 0$, $L_p(\rho,D,s)$ has order of vanishing $e$
at $s=0$, and $D$ satisfies the conditions of \cite[\S 2.2.6]{benois:L-invariant} then
\[\lim_{s\to 0}\frac{L_p(\rho,D,s)}{s^e}=(-1)^e\mathcal{L}(\rho,
D)\mathcal{E}^+(\rho,D)\frac{L(\rho,0)}{\Omega_\infty(\rho)}\]
where $\mathcal{E}^+(\rho,D)$ is defined in \cite[\S 2.3.2]{benois:L-invariant} and $\mathcal{L}(\rho,D)$ is the arithmetic $\mathcal{L}$-invariant defined in \cite[\S 2.2.2]{benois:L-invariant}.
\end{conjecture}

When $F$ is a number field and $\rho:G_F\to \GL(n,
\overline{\QQ}_p)$ is a geometric Galois representation
one may still formulate the above conjecture for
$\Ind_F^{\QQ}\rho$ and a collection of regular submodules $D_v\subset \D_{\st}(\rho|_{G_{F_v}})$ for places $v\mid p$. In that case, it natural to follow Hida and define the arithmetic $\mathcal{L}$-invariant as $\mathcal{L}(\rho,\{D_v\})=\mathcal{L}(\Ind_F^{\QQ}, D)$ where $D=\bigoplus D_v$ is a regular submodule of $\D_{\st}((\Ind_F^{\QQ}\rho)|_{G_{\QQ_p}})$.

We now describe the main results of this article.
Let $F$ be a totally real field in which $p$
splits completely and let $\pi$ be a cohomological Hilbert
modular form (cf.\ \S\ref{sect:hmf def}). Let
$V_{2n}=\Sym^{2n}\rho_{\pi,p}\otimes\det^{-n}\rho_{\pi,p}$ where
$\rho_{\pi,p}$ is the $p$-adic representation of $G_F$ attached
to $\pi$. In the following, when computing
$\mathcal{L}$-invariants of $V_{2n}$, we will assume that the Bloch--Kato Selmer group
$H^1_f(F, V_{2n})$ vanishes and technical condition C4 from \S
\ref{sect:benois-l-inv}. We will also assume that $\pi_v$ is Iwahori
spherical at $v\mid p$. Fixing a basis $e_1,e_2$ for
$\D_{\st}(\rho_{\pi,p}|_{G_{F_v}})$, we fix a regular submodule
$D_v\subset \D_{\st}(V_{2n}|_{G_{F_v}})$ as in \S
\ref{sect:regular submodules}. We remark now that the most difficult case occurs when $V_{2n}$ is crystalline at $v$ and this article provides the first results in this case when $n>3$ (in the semistable cases, the result is either quite easy or, at least in the $p$-ordinary case, follows from the work of Hida \cite{hida:mazur-tate-teitelbaum}).

We begin with the case of $V_2$. Suppose that if $\pi_v$ for
$v\mid p$ is an unramified principal series then the Satake
parameters are distinct. Further suppose that the Hilbert
eigenvariety $\mathcal{E}$ around $\pi$ and the refinement of
$\pi_v$ giving the ordering $e_1,e_2$ is \'{e}tale over the weight
space around the point corresponding to $\pi$ (cf. \S
\ref{sect:sym2l}). Finally, let $a_v$ be the analytic Hecke
eigenvalue corresponding to the double coset $[\Iw
\diag(p,1)\Iw]$.
\begin{theoremletter}[Theorem \ref{t:l invariant formula hmf}]\label{thmA}
Writing $a'_v$ for the derivative in the direction $(1,\ldots, 1; -1)$ in the weight space, we have
\[\mathcal{L}(V_2, \{D_v\}) = \prod_{v\mid p}
\left(\frac{-2a_v'}{a_v}\right)\]
\end{theoremletter}

Such a formula for $p$-ordinary elliptic modular forms was obtained first by Hida \cite{hida:ad0} (under a condition on the Galois deformation ring) and then by the first author \cite{harron:thesis} (under the assumption $H^1_f(\QQ, V_2)$ used here). In \cite{mok:adjoint-L-invariant}, Mok proved this result for finite slope elliptic modular forms.

For higher symmetric powers, we must assume $\pi$ is not CM, so that certain of its symmetric powers are actually cuspidal. When $F=\QQ$, the CM case has been dealt with in \cite{harron:CM} and \cite{harron-lei:CM}.

For the case of $V_6$, we do as in \cite{harron:thesis,harron:sym6-ordinary} and use the Ramakrishnan--Shahidi lift $\Pi$ of $\Sym^3\pi$ to $\GSp(4)$. Suppose that $\pi$ is not CM and that if $\pi_v$ for
$v\mid p$ is an unramified principal series then the ratio of the Satake
parameters is not in $\mu_{60}$ (this condition is necessary for the existence of global triangulations). Further suppose that the genus 2 Siegel--Hilbert
eigenvariety $\mathcal{E}$ around $\Pi$ and the $p$-stabilization of
$\Pi_v$ giving the ordering $e_1,e_2$ is \'{e}tale over the weight
space around the point corresponding to $\Pi$ (cf. \S
\ref{sect:sym6l}). Finally, let $a_{v,1}$ and $a_{v,2}$ be the analytic Hecke
eigenvalues corresponding to the double cosets $[\Iw\diag(1,p^{-1}, p^{-2}, p^{-1})\Iw]$ and $[\Iw\diag(1,1,p^{-1}, p^{-1})\Iw]$.

\begin{theoremletter}[Theorem \ref{t:l invariant formula gsp4}]\label{thmB}
If $\overrightarrow{u}=(u_1,u_2; u_0)$ is any direction in
the weight space, i.e. $u_1\geq u_2\geq 0$,
then
\[\mathcal{L}(V_6, \{D_v\})=\prod_{v\mid p}
\left(\frac{-4\widetilde{\nabla}_{\overrightarrow{u}}a_{v,2}+3\widetilde{\nabla}_{\overrightarrow{u}}a_{v,1}}{u_1-2u_2}\right)\]
where we write $\widetilde{\nabla}_{\overrightarrow{u}}f =
(\nabla_{\overrightarrow{u}}f)/f$ for the logarithmic directional derivative of $f$ evaluated at the point above $\Pi$. 
\end{theoremletter}

This generalizes the main result of \cite{harron:thesis,harron:sym6-ordinary} which compute the arithmetic $\mathcal{L}$-invariant of $V_6$ in the case of $p$-ordinary elliptic modular forms.

The first computation of $\mathcal{L}$-invariants of $V_{2n}$ for general $n$ we present
uses symplectic eigenvarieties and is, for now, conditional on
the stabilization of the twisted trace formula (this is
necessary for the construction of an analytic Galois
representation). Suppose $\pi$ is not CM, and that for $v\mid p$ such that
$\pi_v$ is unramified the ratio of the Satake parameters is not in 
$\mu_{\infty}$ (again, necessary for the existence of global triangulations). Suppose $\pi$ satisfies the hypotheses of Theorem \ref{t:automorphy of sym} (2) and let
$\Pi$ be a suitable twist of the cuspidal representation of $\GSp(2n, \AA_F)$ from Theorem \ref{t:GL(2n+1) to Sp(2n)}. Let $\mathcal{E}$ be Urban's eigenvariety for $\GSp(2n)$ and let $a_{v,i}$ be the analytic Hecke eigenvalues from the proof of Lemma \ref{l:gsp eigenvariety}. Suppose that the eigenvariety $\mathcal{E}$ is \'{e}tale over the weight space at the $p$-stabilization of $\Pi$ corresponding to the ordering $e_1,e_2$. 

\begin{theoremletter}[Theorem \ref{t:l invariant formula gsp}]\label{thmC}
If $\overrightarrow{u}=(u_1,\ldots,u_n;u_0)$ is any direction in
the weight space, then
\[\mathcal{L}(V_{4n-2}, \{D_v\})=\prod_{v\mid
  p}-\left(\frac{B_{n}\widetilde{\nabla}_{\overrightarrow{u}}a_{v,1}+B_{1}(\widetilde{\nabla}_{\overrightarrow{u}}a_{v,n-1}-2\widetilde{\nabla}_{\overrightarrow{u}}a_{v,n})+\sum_{i=2}^{n-1}B_{i}(\widetilde{\nabla}_{\overrightarrow{u}}a_{v,i-1}-\widetilde{\nabla}_{\overrightarrow{u}}a_{v,i})}{\sum_{i=1}^nu_iB_{n+1-i}}\right)\]
where we write $\displaystyle B_i=(-1)^i\binom{2n}{n+i}i$.
\end{theoremletter}
As mentioned earlier, the results of Hida \cite{hida:mazur-tate-teitelbaum} address the case where $V_{2n}$ is ordinary and semistable but not crystalline.

Our second computation of $\mathcal{L}$-invariants for $V_{2n}$ uses unitary groups, is also conditional on the stabilization of the twisted trace formula, and is more restrictive. It however has the advantage that work in progress of Eischen--Harris--Li--Skinner will provide several-variable $p$-adic $L$-functions for Hida families on unitary groups and thus the second step of the Greenberg--Stevens method may be closer at hand. As above suppose $\pi$ is not CM, and that for $v\mid p$ such that
$\pi_v$ is unramified the ratio of the Satake parameters is not in 
$\mu_{\infty}$. Assume Conjecture \ref{t:GL(n) to U_n} and suppose $\pi$ satisfies the hypotheses of
Theorem \ref{t:automorphy of sym} (2) and Proposition \ref{p:sym hmf to unitary} (the latter requires $\pi_v$ to be special at two finite places not above $p$). Let $E/F$ the CM extension and $\Pi$ the cuspidal representation of $U_{4n}(\AA_F)$ to which there is a transfer of a twist of $\Sym^{4n-1}\pi$ as in Proposition \ref{p:sym hmf to unitary}. Let $\mathcal{E}$ be Chenevier's eigenvariety and let $a_{v,i}$ be the analytic Hecke eigenvalues from the proof of Corollary \ref{c:unitary triangulation}. Suppose that the eigenvariety $\mathcal{E}$ is \'{e}tale over the weight space at the $p$-stabilization of $\Pi$ corresponding to the ordering $e_1,e_2$. 

\begin{theoremletter}[Theorem \ref{t:l invariant formula unitary}]\label{thmD}
If $\overrightarrow{u}=(u_1,\ldots,u_n;u_0)$ is any direction in
the weight space, then
\[\mathcal{L}(V_{8n-2}, \{D_v\})=\prod_{v\mid
  p}\left(\frac{-\sum_{i=1}^{4n}(-1)^{i-1}\binom{4n-1}{i-1}\widetilde{\nabla}_{\overrightarrow{u}}a_{v,i}}{\sum_{i=1}^{4n}(-1)^{i-1}\binom{4n-1}{i-1}u_i}\right)\]
and
\[\mathcal{L}(V_{8n-6}, \{D_v\})=\prod_{v\mid
  p}\left(\frac{-\sum_{i=1}^{4n}B_{i-1}\widetilde{\nabla}_{\overrightarrow{u}}a_{v,i}}{\sum_{i=1}^{4n}u_iB_{i-1}}\right).\]
Here $B_i=B_{4n-1,4n-3,i}$ is the inverse Clebsch--Gordan
coefficient of Proposition \ref{p:End to Sym^n}, up to a scalar independent of $i$ given by
\[B_i= (-1)^i\binom{4n-1}{i}((4n-1)^3-(4i+1)(4n-1)^2+(4i^2+2i)(4n-1)-2i^2).\]
\end{theoremletter}

\begin{remark}
The assumption that $\pi$ satisfy the hypotheses of
Theorem \ref{t:automorphy of sym} (2), i.e.\ that various
$\Sym^{k}\pi$ be automorphic over $F$ is necessary for our computations. Ongoing work of Clozel and Thorne provides the automorphy of such symmetric powers when $k$ is small.
\end{remark}

The paper is organized as follows. In Section \ref{sect:1}, we
describe Benois' definition of the arithmetic
$\mathcal{L}$-invariant and triangulations in $p$-adic families.
In Section \ref{sect:hmf}, we study Galois representations
attached to Hilbert modular forms and functorial transfers to
unitary and symplectic groups. Then, in Section \ref{sect:eigs}, we describe the unitary and symplectic eigenvarieties and global triangulations of certain analytic Galois representations. In section \ref{sect:l}, we prove Theorems \ref{thmA} through \ref{thmD}. Finally, the appendix discusses some plethysm for $\GL(2)$, relating the $B_{n,k,i}$ to inverse Clebsch--Gordon coefficients and proving an explicit formula for them.


\section*{Some basic notation}\label{sect:notation}

Throughout this article, $p$ denotes a fixed odd prime. We will use $L$ to denote a finite extension of $\Qp$. By a $p$-adic representation of a group $G$, we mean a continuous homomorphism $\rho:G\rightarrow\GL(V)$, where $V$ is a finite-dimensional vector space over $L$. Let $\mu_{p^\infty}$ denote the set of $p$-power roots of unity and let $\chi$ denote the $p$-adic cyclotomic character giving the action of whatever appropriate Galois group on it. Let $\Gamma$ denote the Galois group $\Gal(\Qp(\mu_{p^\infty})/\Qp)\cong\Zp^\times$ and let $\gamma$ denote a fixed topological generator. If $F$ is a field, $G_F$ denotes its absolute Galois group and $H^\bullet(F,-)=H^\bullet(G_F,-)$ denotes it Galois cohomology.


\section{\texorpdfstring{$(\varphi,\Gamma)$}{(phi, Gamma)}-modules and $\mathcal{L}$-invariants}\label{sect:1}

This section contains a review of some pertinent content from \cite{benois:L-invariant} and \cite{liu:triangulation}. We refer the reader to these articles for further details.

\subsection{\texorpdfstring{$(\varphi,\Gamma)$}{(phi, Gamma)}-modules over the Robba ring}

For a real number $r$ with $0\leq r<1$ denote by $\mathcal{R}^r$ the set of power
series $f(x)=\sum_{k\in \ZZ}a_kx^k$ holomorphic for $r\leq |x|_p<1$ with $a_k\in\Qp$. Let $\mathcal{R}=\bigcup_{r<1}\mathcal{R}^r$ be the \textit{Robba ring} (over $\Qp$).

The Robba ring carries natural actions of a Frobenius, $\varphi$, and
$\Gamma$ given as follows. If $f\in \mathcal{R}$, then
\[(\varphi f)(x)=f((1+x)^p-1)\]
and for $\tau \in \Gamma$,
\[(\tau f)(x)=f((1+x)^{\chi(\tau)}-1).
\]

More generally, if $S$ is an affinoid algebra over $\Qp$, define $\mathcal{R}_{S}=\mathcal{R}\widehat{\otimes}_{\Qp} S$ (cf.~\cite[Proposition 2.1.5]{liu:triangulation}). Extending the actions of $\varphi\otimes 1$ and $\tau\otimes1$ linearly, we get actions of $(\varphi,\Gamma)$ on $\mathcal{R}_S$; i.e.\ this should be thought of as the Robba ring over $\Qp$ with coefficients in $S$.

A $(\varphi, \Gamma)$-module over $\mathcal{R}_L$ is a free
$\mathcal{R}_L$-module $D_L$ of finite rank, equipped with
a $\varphi$-semilinear Frobenius map $\varphi_{D_L}$ and a semilinear action of $\Gamma$ which commute with
each other, such that the induced map
$\varphi_{D_L}^*D_L=D_L\otimes_{\varphi}\mathcal{R}_L\to D_L$ is an
isomorphism. More generally, a $(\varphi, \Gamma)$-module over
$\mathcal{R}_S$ is a vector bundle $D_S$ (coherent,
locally free sheaf) over $\mathcal{R}_S$ of finite rank,
equipped with a semilinear Frobenius, $\varphi_{D_S}$, and a semilinear action of $\Gamma$,
commuting with each other, such that $\varphi_{D_S}^\ast D_S\to
D_S$ is an isomorphism.

If $\delta:\Qp^\times\to S^\times$ is a continuous
character, define $\mathcal{R}_S(\delta)$ as the rank 1
$(\varphi,\Gamma)$-module $\mathcal{R}_S e_\delta$ with basis
$e_\delta$ with $\varphi_{\mathcal{R}_S(\delta)}(x
e_\delta)=\varphi_{\mathcal{R}_S}(x) \delta(p)e_\delta$ and for
$\tau \in \Gamma$, $\tau(x e_\delta) =
\tau(x)\delta(\chi(\tau))e_\delta$.

A $(\varphi,\Gamma)$-module $D_S$ is said to be \textit{trianguline} if
there exists an increasing (separated, exhaustive) filtration
 $\Fil_\bullet D_S$ such that the graded pieces are of the form
 $\mathcal{R}_S(\delta)\otimes_SM$ for some continuous
 $\delta:\Qp^\times\to S^\times$ and locally free one-dimensional
 $M$ over $S$ with trivial $(\varphi,\Gamma)$ actions.

 There is a functor $D_{\rig,L}^\dagger$ associating to an
 $L$-linear continuous representation of $G_{\QQ_p}$ a
 $(\varphi,\Gamma)$-module over $\mathcal{R}_{L}$ and more
 generally a functor $D_{\rig,S}^\dagger$ associating to an
 $S$-linear continuous $G_{\QQ_p}$-representation a
 $(\varphi, \Gamma)$-module over $\mathcal{R}_S$. The functor
 $D_{\rig,L}^\dagger$ induces an isomorphism of categories between
 the category of $L$-linear continuous representations of
 $G_{\QQ_p}$ and slope 0 $(\varphi,\Gamma)$-modules over
 $\mathcal{R}_L$.

There exist functors $\mathcal{D}_{\cris}$ (resp.\ $\mathcal{D}_{\st}$) attaching to a $(\varphi,\Gamma)$-module $D$ over $\mathcal{R}_L$ a filtered $\varphi$-module (resp.\ $(\varphi, N)$\mbox{-module)} over $\Qp$ with coefficients in $L$ such that if $V$ is crystalline (resp.\ semistable) then $\mathcal{D}_{\cris}(D_{\rig,L}^\dagger(V))\cong D_{\cris}(V)$ (resp.\ $\mathcal{D}_{\st}(D_{\rig,L}^\dagger(V))\cong D_{\st}(V)$).

Suppose $V$ is a finite-dimensional $L$-linear continuous
representation of $G_{\QQ_p}$ and
$D=D_{\rig,L}^\dagger(V)$ is the associated
$(\varphi,\Gamma)$-module over $\mathcal{R}_L$. It is a theorem
of Ruochuan Liu (\cite{liu:herr}) that the Galois cohomology $H^\bullet(\QQ_p, V)$
can be computed as the cohomology $H^\bullet(D)$ of the Herr complex
 $0\to D\stackrel{f}{\to} D\oplus D\stackrel{g}{\to} D\to 0$
 where the transition maps are $f(x) = (\varphi_D-1)x\oplus
 (\gamma-1)x$ and $g(x,y) = (\gamma-1)x-(\varphi_D-1)y$. We denote by $\cl(x,y)$ the image of $x\oplus y\in D\oplus D$ in $H^1(D)$.

The Bloch--Kato local conditions $H_f^1(\QQ_p, V)=\ker \left(H^1(\QQ_p, V)\to H^1(\QQ_p, V\otimes_{\Qp} B_{\cris})\right)$ and $H^1_{\st}(\QQ_p,V)$ (analogously defined) can also be computed directly using Herr's complex.  
Indeed, if $\alpha=a\oplus b\in D\oplus D$ one gets an extension $D_\alpha=D\oplus \mathcal{R}_L e$, depending only on $cl(a,b)$, endowed with Frobenius and $\Gamma$-action defined by $(\varphi_D-1)(0\oplus e) = a\oplus 0$ and $(\gamma-1)(0\oplus e) = b\oplus 0$. Let $H_f^1(D)$ be the set of crystalline extensions $D_\alpha$, i.e.\ those satisfying $\dim_{\QQ_p}\mathcal{D}_{\cris}(D_\alpha)=\dim_{\QQ_p}\mathcal{D}_{\cris}(D)+1$ and $H_{\st}^1(D)$ be the set of semistable extensions (defined analogously). Then 
\begin{align*}
H^1_f(\QQ_p,V)\cong H^1_f(\Drig(V))\quad\textrm{and}\quad H^1_{\st}(\QQ_p,V)\cong H^1_{\st}(\Drig(V)).
\end{align*}

We end with the following computation of Benois
(\cite[Proposition 1.5.9]{benois:L-invariant}). Let
$\delta:\QQ_p^\times\to L^\times$ be the character
$\delta(x)=x^{-k}$ where $k\in \ZZ_{\geq 0}$. The rank 1 module $\mathcal{R}_L(\delta)$ is crystalline and so $\mathcal{D}_{\cris}(\mathcal{R}_L(\delta))=\mathcal{R}_L(\delta)^{\Gamma}\subset \mathcal{R}_L(\delta)$. This allows us to define the map
\[i:\mathcal{D}_{\cris}(\mathcal{R}_L(\delta))\oplus
\mathcal{D}_{\cris}(\mathcal{R}_L(\delta))\to
H^1(\mathcal{R}_L(\delta))\]
by $i(x,y)=\cl(-x, y\log\chi(\gamma))$. Then $i$ is an isomorphism and $H^1_f(\mathcal{R}_L(\delta))\cong i(\mathcal{D}_{\cris}(\mathcal{R}_L(\delta))\oplus 0)$. Moreover, defining $H^1_c(\mathcal{R}_L(\delta))=i(0\oplus \mathcal{D}_{\cris}(\mathcal{R}_L(\delta)))$, both $H^1_f(\mathcal{R}_L(\delta))$ and $H^1_c(\mathcal{R}_L(\delta))$ have rank 1.


\subsection{Regular submodules and the Greenberg--Benois $\mathcal{L}$-invariant}\label{sect:benois-l-inv}\label{sect:regular submodules general}

(cf.\ \cite[\S\S2.1--2.2]{benois:L-invariant}) Let $\rho:G_\QQ\rightarrow\GL(V)$ be a ``geometric'' $p$-adic representation of $G_\QQ$, i.e.\ $\rho$ is unramified outside a finite set of places and it is potentially semistable at $p$. Let $S$ be a finite set of places containing the ramified ones as well as $p$ and $\infty$. We give an overview of Benois' definition of the arithmetic $\mathcal{L}$-invariant of $V$. His definition requires five additional assumptions (C1--5) on $V$ described below. We also discuss $\mathcal{L}$-invariants of representations of $G_F$, where $F$ is a number field. We end by proving a lemma we later use in our computations of $\mathcal{L}$-invariants.

For $\ell\nmid p\infty$, define $H^1_f(\QQ_\ell,V)=\ker(H^1(\QQ_\ell,V)\rightarrow H^1(I_\ell,V))$, where $I_\ell$ denotes the inertia subgroup of $G_{\QQ_\ell}$. When $\ell=p$, the Bloch--Kato local condition $H^1_f(\QQ_p,V)$ was defined in the previous section. Finally, let $H^1_f(\RR,V)=H^1(\RR,V)$.
Let $G_{\QQ,S}=\Gal(\QQ_S/\QQ)$, where $\QQ_S$ is the maximal extension of $\QQ$ unramified outside of $S$. Define the Bloch--Kato Selmer group of $V$ as
\[ H^1_f(V)=\ker\left(H^1(G_{\QQ,S}, V)\to \bigoplus_{v\in
  S}H^1(\QQ_v,V)/H^1_f(\QQ_v, V)\right),
\]
which does not depend on the choice of $S$.


As in \cite[\S 2.1.2]{benois:L-invariant} we will assume:
\begin{enumerate}
\item[(C1)]\label{assumption:C1}$H_f^1(V)=H^1_f(V^*(1))=0$;
\item[(C2)]\label{assumption:C2}$H^0(G_{\QQ ,S},V)=H^0(G_{\QQ ,S},V^*(1))=0$;
\item[(C3)]\label{assumption:C3}$V|_{G_{\QQ_p}}$ is semistable and the
semistable Frobenius $\varphi$ is semisimple at $1$ and
$p^{-1}$;
\item[(C4)]\label{assumption:C4}$\Drig(V|_{G_{\QQ_p}})$ has no saturated subquotient isomorphic to some $U_{k,m}$ for $k\geq 1$ and $m\geq 0$ (cf. \cite[\S 2.1.2]{benois:L-invariant}, where $U_{k,m}$ is the (unique) non-split crystalline $(\varphi,\Gamma)$-module extension of $\mathcal{R}(x^{-m})$ by $\mathcal{R}(|x|x^k)$).
\end{enumerate}

A \textit{regular submodule} $D$ of $\D_{\st}(V)$ is a
$(\varphi,N)$-submodule such that $D\cong
\D_{\st}(V)/\Fil^0\D_{\st}(V)$ (as vector spaces) under the natural projection
map. Given a regular submodule $D$, Benois constructs the filtration
\begin{align*}
D_{-1} &= (1-p^{-1}\varphi^{-1})D+N(D^{\varphi=1})\\
D_0&=D\\
D_1&=D+\D_{\st}(V)^{\varphi=1}\cap N^{-1}(D^{\varphi=p^{-1}})
\end{align*}

The filtration $D_\bullet$ on $\D_{\st}(V)$ gives a
filtration $F_\bullet\Drig(V)$ by setting
\[F_i\Drig(V)=\Drig(V)\cap (D_i\otimes_{\Qp}
\mathcal{R}_L[1/t])\]
(here $t=\log(1+x)\in \mathcal{R}_L$).

Define the \textit{exceptional subquotient} of $V$ to be $W=F_1\Drig(V)/F_{-1}\Drig(V)$, which is a $(\varphi,\Gamma)$ analogue of Greenberg's $F^{00}/F^{11}$ (see \cite[p.~157]{G94}). Benois shows there are unique decompositions
\begin{align*}
W&\cong W_0\oplus W_1\oplus M\\
\gr_0\Drig(V)&\cong W_0\oplus M_0\\
\gr_1\Drig(V)&\cong W_1\oplus M_1
\end{align*}
such that $W_0$ has rank $\dim H^0(W^*(1))$, $W_1$ has rank $\dim H^0(W)$. Moreover, $M_0$ and $M_1$ have equal rank and the sequence $0\to M_0\stackrel{f}{\to}M\stackrel{g}{\to}M_1\to 0$ is exact.

One has
\[H^1(W)=\coker(H^1(F_{-1}\Drig(V))\rightarrow H^1(F_1\Drig(V))),
\]
\[H^1_f(W)=\coker(H^1_f(F_{-1}\Drig(V))\rightarrow H^1_f(F_1\Drig(V))),
\]
and $H^1(W)/H^1_f(W)$ has dimension $e_D=\rk M_0+\rk W_0+\rk W_1$.

The (dual of the) Tate-Poitou exact sequence
gives an exact sequence
\[0\to H^1_f(V)\to H^1(G_{\QQ ,S},V)\to \displaystyle \bigoplus_{v\in
  S}\frac{H^1(\QQ _v,V)}{H^1_f(\QQ _v,V)}\to H^1_f(V^*(1))^\vee\]
where $V^*=\Hom(V, \QQ_p)$ and $A^\vee=\Hom(A,
\QQ/\ZZ)$. Assumptions (C1) and (C2) above imply that
\[H^1(G_{\QQ ,S},V)\cong \displaystyle \bigoplus_{v\in
  S}\frac{H^1(\QQ _v,V)}{H^1_f(\QQ _v,V)}\]
Note that $\displaystyle \bigoplus_{v\in
  S}\frac{H^1(\QQ _v,V)}{H^1_f(\QQ _v,V)}$ contains the $e_D$-dimensional subspace $\displaystyle \frac{H^1(W)}{H^1_f(W)}\cong \frac{H^1(F_1\Drig(V_p))}{H^1_f(\QQ_p,V_p)}$. Define
  $H^1(D,V)\subset H^1(G_{\QQ ,S},V)$ to be the set of classes whose image in
$\displaystyle \frac{H^1(\QQ_p,V)}{H^1_f(\QQ_p,V)}$ lies in $\displaystyle \frac{H^1(W)}{H^1_f(W)}$.

From now on, assume
\begin{itemize}
\item[(C5)] $W_{0}=0$.
\end{itemize}

Since $\varphi$ acts as 1 on $\gr_1\Drig(V)$, \cite[Proposition 1.5.9]{benois:L-invariant} implies, assuming that the Hodge--Tate weights are nonnegative, that $\gr_1\Drig(V)\cong \oplus \mathcal{R}_L(x^{-k_i})$ where the $k_i\geq 0$ are the Hodge--Tate weights. Thus one obtains a decomposition $H^1(\Drig(V))\cong H^1_f(\Drig(V))\oplus H^1_c(\Drig(V))$ by summing the decompositions for each $\mathcal{R}_L(x^{-k_i})$; as in the rank 1 case, $H^1_f(\gr_1\Drig(V))\cong H^1_c(\gr_1\Drig(V))\cong \mathcal{D}_{\cris}(\gr_1\Drig(V))$.  There exist linear maps $\rho_{D,?}:H^1(D,V)\to\mathcal{D}_{\cris}(\gr_1\Drig(V))$ for $?\in \{f,c\}$ making the following diagram commute:
\[\xymatrix{
  \mathcal{D}_{\cris}(\gr_1\Drig(V))\ar[r]^\cong_{\iota_f}&H^1_f(\gr_1\Drig(V))\\
  H^1(D,V)\ar[r]\ar[u]^{\rho_{D,f}}\ar[d]_{\rho_{D,c}}&H^1(\gr_1\Drig(V))\ar[u]_{\pi_f}\ar[d]^{\pi_c}\\
  \mathcal{D}_{\cris}(\gr_1\Drig(V))\ar[r]^\cong_{\iota_c}&H^1_c(\gr_1\Drig(V))
}\]
Under the assumption that $W_0=0$, Benois shows that the linear map $\rho_{D,c}$ is invertible and defines the arithmetic $\mathcal{L}$-invariant as
\[\mathcal{L}(V,D):={\det}\left(\rho_{D,f}\circ\rho_{D,c}^{-1}\big|\mathcal{D}_{\cris}(\gr_1\Drig(V))\right).\]

In the case of a number field $F$ and a $p$-adic
representation $\rho:G_F\to \GL(V)$ we
follow Hida in defining an arithmetic
$\mathcal{L}$-invariant. Suppose that $p$ splits completely in
$F$, that $\rho$ is unramified almost everywhere and that at
all $v\mid p$, $\rho|_{G_{F_v}}$ is semistable. Assume that $\rho$ satisfies conditions (C1--4) (with appropriate modifications). Then the representation $\Ind_F^{\QQ}\rho$
satisfies conditions (C1--4). Indeed, conditions (C3)
and (C4) follow from the fact that
\[(\Ind_F^{\QQ}\rho)|_{G_{\QQ_p}}\cong
\bigoplus_{v\mid p}\Ind_{F_v}^{\QQ_p}(\rho|_{G_{F_v}})\cong
\bigoplus_{v\mid p}\rho|_{G_{F_v}}\]
since $p$ splits completely in
$F$. Conditions (C1) and (C2) follow from Shapiro's lemma.

For $v\mid p$, let $V_v:=V|_{G_{F_v}}$. Choose $D_v\subset
\D_{\st}(V_v)$ a regular submodule giving the modules $W_{0,v}$,
$W_{1,v}$ and $M_v$. Then $D=\oplus_{v\mid p}D_v\subset
\oplus_{v\mid p}\D_{\st}(\rho_v)\cong
\D_{\st}((\Ind_F^{\QQ}\rho)|_{G_{\QQ_p}})$ is a
regular submodule and $W_0=\oplus_{v\mid p}W_{0,v}$,
$W_1=\oplus_{v\mid p}W_{1,v}$ and $M=\oplus_{v\mid
  p}M_v$. Assuming $W_{0,v}=0$ for every $v\mid p$ yields
$W_0=0$ and we may define
\[\mathcal{L}(\{D_v\}, \rho)=\mathcal{L}(D, \Ind_F^{\QQ}\rho)\]
Note that $\gr_1\Drig((\Ind_F^{\QQ}\rho)|_{G_{\QQ_p}})\cong \bigoplus_{v\mid p}\gr_1\Drig(V_v)$.

\begin{lemma}\label{l:rank 1 L invariant}
Suppose that $\gr_1\Drig(V_v)\cong \mathcal{R}$ for each $v\mid
p$. Let $H^1(\mathcal{R})\cong H^1_f(\mathcal{R})\oplus
H^1_c(\mathcal{R})$ with basis $x=(-1,0)$ and $y=(0,\log_p
\chi(\gamma))$. Suppose $c\in
H^1(D, \Ind_F^{\QQ}\rho)$ is such that the image of $c$
in $H^1(\gr_1\Drig(V_v))$ is $\xi_v=a_vx+b_vy$ with $b_v\neq
0$. Then
\[\mathcal{L}(\{D_v\},\rho)=\prod_{v\mid p}\frac{a_v}{b_v}\]
\end{lemma}
\begin{proof}
Since
\[H^1(\gr_1\Drig((\Ind_F^{\QQ}\rho)|_{G_{\QQ_p}}))\cong
\bigoplus_{v\mid p}H^1(\gr_1\Drig(\rho_v))
\]
and
\[
\mathcal{D}_{\cris}(\gr_1\Drig((\Ind_F^{\QQ}\rho)|_{G_{\QQ_p}}))\cong
\bigoplus_{v\mid p}\mathcal{D}_{\cris}(\gr_1\Drig(\rho_v)),
\]
the maps $\iota_f$, $\iota_c$, $\pi_f$, and $\pi_c$ are direct sums
of the maps
\[
\begin{array}{c}
\iota_{f,v}:\mathcal{D}_{\cris}(\mathcal{R})\cong
H^1_{f}(\mathcal{R}),\quad
\iota_{c,v}:\mathcal{D}_{\cris}(\mathcal{R})\cong
H^1_{c}(\mathcal{R}),\\
\\
\pi_{f,v}:H^1(\mathcal{R})\to H^1_f(\mathcal{R})\quad\text{and}\quad
\pi_{c,v}:H^1(\mathcal{R})\to H^1_c(\mathcal{R}).
\end{array}
\]
Let
$\xi_v=a_vx+b_vy$ be the image of $c$ in
$H^1(\gr_1\Drig(\rho_v))\cong H^1(\mathcal{R})$. We deduce that
the maps $\rho_f$ (resp.\ $\rho_c$) are direct sums of the $\rho_{f,v}=\iota_{f,v}^{-1}\circ\pi_{f,v}$ (resp.\ $\rho_{c,v}=\iota_{c,v}^{-1}\circ\pi_{c,v}$). Then $a_v\in
\mathcal{D}_{\cris}(\gr_1\Drig(\rho_v))$ is the image of $\xi_v$
under $\rho_f$ and $b_v\in
\mathcal{D}_{\cris}(\gr_1\Drig(\rho_v))$ is the image of $\xi_v$
under $\rho_c$. We deduce that
\begin{align*}
\mathcal{L}(\{D_v\}, V)&=\det(\rho_f\circ\rho_c^{-1}|\mathcal{D}_{\cris}(\gr_1\Drig((\Ind_F^{\QQ}\rho)|_{G_{\QQ_p}})))\\
&=\prod_{v\mid p}\det(\rho_{f,v}\circ\rho_{c,v}^{-1}|\mathcal{D}_{\cris}(\gr_1\Drig(\rho_v)))\\
&=\prod_{v\mid p}\frac{a_v}{b_v}
\end{align*}

\end{proof}

\subsection{Refined families of
  Galois representations}\label{sect:refined galois}
Lemma \ref{l:rank 1 L invariant} provides a framework for computing arithmetic $\mathcal{L}$-invariants as long as one is able to produce cohomology classes $c\in H^1(D, \Ind_F^{\QQ}\rho)$ and compute their projections to $H^1(\gr_1\Drig(\rho_v))$. We will produce such cohomology classes using analytic Galois representations on eigenvarieties and we will compute explicitly the projections (in effect the $a_v$ and $b_v$ of Lemma \ref{l:rank 1 L invariant}) using global triangulations of $(\varphi, \Gamma)$-modules.

We recall here the main result of \cite[\S 5.2]{liu:triangulation} on triangularization in
refined families. Let $L$ be a finite extension of $\QQ_p$. Suppose $X$ is a separated and reduced rigid analytic
space over $L$ and $V_X$ is a locally free, coherent
$\mathcal{O}_X$-module of rank $d$ with a continuous $\mathcal{O}_X$-linear
action of the Galois group $G_{\QQ_p}$. The family $V_X$
of Galois representations is said to be \textit{refined} if there exist $\kappa_1,\ldots,\kappa_d\in \mathcal{O}(X)$, $F_1,\ldots,F_d\in \mathcal{O}(X)$, and a Zariski dense set of points $Z\subset X$ such that:
\begin{enumerate}
\item for $x\in X$, the Hodge--Tate weights of $\D_{\dR}(V_{x})$ are
$\kappa_1(x),\ldots,\kappa_d(x)$,
\item for $z\in Z$, the representation $V_z$ is crystalline,
\item for $z\in Z$, $\kappa_1(z)<\ldots <\kappa_d(z)$,
\item the eigenvalues of $\varphi$ acting on $\D_{\cris}(V_z)$
are distinct and equal to $\{p^{\kappa_i(z)}F_i(z)\}$,
\item \label{cond:5}for $C\in\ZZ_{\geq0}$ the set $Z_C$, consisting of $z\in Z$ such that
for all distinct subsets $I,J\subset \{1,\ldots,d\}$ of equal
cardinality one has $|\sum_{i\in I}\kappa_i(z)-\sum_{j\in
  J}\kappa_j(z)|>C$, accumulates at every $z\in Z$,
\item for each $1\leq i\leq d$, there exists a continuous character $\chi_i:\mathcal{O}_K^\times\to \mathcal{O}(X)^\times$ such that $\chi_{z,i}'(1)=\kappa_i$ and $\chi_{z,i}(u)=u^{\kappa_i(z)}$ for all $z\in Z$.
\end{enumerate}
Given a refined family $V_X$, we define
$\Delta_i:\Qp^\times\to \mathcal{O}(X)^\times$ by
$\Delta_i(p)=\prod_{j=1}^i F_j$ and for $u\in
\Zp^\times$, $\Delta_i(u)=\prod_{j=1}^i\chi_j(u)$. Let
$\delta_i=\Delta_i/\Delta_{i-1}$.

Let $V_X$ be a refined family. For all $z\in Z$, there is an induced \textit{refinement} of $V_z$, i.e.\ a filtration $0=\mathcal{F}_0\subsetneq \mathcal{F}_1\subsetneq\ldots\subsetneq \mathcal{F}_d=\D_{\cris}(V_z)$ of $\varphi$-submodules. It is determined by the condition that the eigenvalue of $\varphi_{\cris}$ on $\mathcal{F}_i/\mathcal{F}_{i-1}$ is $p^{\kappa_i(z)}F_i(z)$. We say that $z$ is \textit{noncritical} if $\D_{\cris}(V_z)=\mathcal{F}_i+\Fil^{\kappa_{i+1}(z)}\D_{\cris}(V_z)$. We say that $z$ is \textit{regular} if $\det\varphi$ on $\mathcal{F}_i$ has multiplicity one in $\D_{\cris}(\wedge^i V_z)$ for all $i$.

Then, \cite[Theorem 5.2.10]{liu:triangulation} gives:
\begin{theorem}\label{t:triangulation}
If $z\in Z$ is regular and noncritical, then in an affinoid neighborhood $U$ of $z$, $V_{U}$ is trianguline with graded pieces isomorphic to $\mathcal{R}_{U}(\delta_1),\ldots,\mathcal{R}_{U}(\delta_d)$.
\end{theorem}

\begin{remark}
In fact, \cite[Theorem 5.3.1]{liu:triangulation} shows that $V_x$ is trianguline at all $x\in X$, but the triangulation is only made explicit over a proper birational transformation of $X$.
\end{remark}

\section{Symmetric powers of Hilbert modular forms}\label{sect:hmf}
\subsection{Hilbert modular forms and their Galois representations}\label{sect:hmf def}
Let
$F/\QQ$ be a totally real field of degree $d$ and let
$I\subseteq\Hom_{\QQ}(F, \CC)$ be a parametrization of the
infinite places. We fix an embedding $\iota_\infty:\overline{\QQ}\into \CC$ thus identifying $I$ with a subset of $\Hom_{\QQ}(F,\ol{\QQ})$. We also fix $\iota_p:\overline{\QQ}\into \overline{\QQ}_p$ which identifies $\Hom_{\QQ}(F,\ol{\QQ}_p)$ with $\Hom_{\QQ}(F,\ol{\QQ})$. This determines a partition $I=\bigcup_{v\mid p}I_v$. Let $\varpi_v$ be a uniformizer for $F_v$ and $e_v$ be the ramification index of $F_v/\QQ_p$. 

Before defining Hilbert modular forms, we need some notation on
representations of $\GL(2,\RR)$. Recall that the Weil
group of $\RR$ is $W_{\RR}=\CC^\times\rtimes \{1,j\}$ where
$j^2=-1$ and $jz=\overline{z}j$ for $z\in \CC$. For an
integer $n\geq 2$, let $\mathcal{D}_n$ be the essentially
discrete series representation of $\GL(2,\RR)$ whose Langlands parameter is $\phi_n:W_{\RR}\to \GL(2, \CC)$ given by $\phi_n(z) = \begin{pmatrix} (z/\overline{z})^{(n-1)/2}&\\ &(\overline{z}/z)^{(n-1)/2}\end{pmatrix}$ and $\phi_n(j) = \begin{pmatrix} &1\\ (-1)^{n-1}&\end{pmatrix}$. The representation $\mathcal{D}_n$ is unitary and has central character $\sign^n$ and Blattner parameter $n$.
More generally, if $t\in \CC$, the representation
$\mathcal{D}_n\otimes |\det|^t$ has associated Langlands
parameter $\phi_n\otimes |\cdot|^{2t}$. If $w\in \ZZ$,
then
\[H^1(\mathfrak{gl}_2,\SO(2),\mathcal{D}_n(-w/2)\otimes
V_{(w+n-2)/2, (w-n+2)/2}^\vee)\neq 0\]
where $V_{(a,b)} = \Sym^{a-b}\CC^2\otimes \det^b$ is the representation of highest weight $(a,b)$.

By a cohomological Hilbert modular form of infinity type $(k_1, \ldots, k_d, w)$, we mean a cuspidal automorphic representation $\pi$ of $\GL(2, \AA_F)$ such that
\begin{enumerate}
\item $k_i\equiv w\pmod{2}$ with $k_i\geq 2$ and 
\item for every $i\in I$, $\pi_i\cong \mathcal{D}_{k_i}\otimes |\det|^{-w/2}$.
\end{enumerate}
This is equivalent to the fact that $\pi_i$ has Langlands
parameter $z\mapsto |z|^{-w}\begin{pmatrix}
(z/\overline{z})^{(k_i-1)/2}&\\ &
(z/\overline{z})^{-(k_i-1)/2}\end{pmatrix}$,
$j\mapsto \begin{pmatrix} &1\\
(-1)^{k_i-1}&\end{pmatrix}$. Since $H^1(\mathfrak{gl}_2,\SO(2),\pi_i\otimes
V_{(-w+k_i-2)/2, (-w-k_i+2)/2}^\vee)\neq 0$, the representation $\pi$ can realized in the cohomology of the local system $\left(\bigotimes_i\left(\Sym^{k_i-2}\otimes \det^{(-w-k_i+2)/2}\right)\right)^\vee$ over a suitable Hilbert modular variety (cf. \cite[\S 3.1.9]{raghuram-tanabe:hilbert}). When $F=\QQ$ and $w=2-k$ we recover the usual notion of an elliptic modular form of weight $k$.

If $p$ is a prime number, then (by Eichler, Shimura, Deligne, Wiles, Taylor, Blasius--Rogawski) there exists a continuous $p$-adic Galois representation $\rho_{\pi,p}:G_F\to \GL(2, \overline{\QQ}_p)$ such that $L^S(\pi, s-1/2)=L^S(\rho_{\pi,p},s)$ for a finite set $S$ of places of $F$. Moreover, one has local-global compatibility: if $v\in S$, then $\WD(\rho_{\pi,p}|_{G_{F_v}})^{\Frss}\cong\rec(\pi_v\otimes|\cdot|^{-1/2})$. When $v\nmid p$ this follows from the work of Carayol (\cite{carayol:hilbert-ell}) and when $v\mid p$ from the work of Saito (\cite{saito:hilbert-p}) and Skinner (\cite{skinner:hilbert}). Finally, the Hodge--Tate weights of $\rho_{\pi,p}|_{G_{F_v}}$ are $(w-k_v)/2$ and $(w+k_v-2)/2$. (For weight $k$ elliptic modular forms with $w=2-k$ this amounts to Hodge--Tate weights $1-k$ and $0$.)

We end this discussion with the following result on the
irreducibility of symmetric powers.
\begin{lemma}\label{l:hmf lie irreducible}
Suppose $\pi$ is not CM. Then $\rho_{\pi,p}$ and $\Sym^n\rho_{\pi,p}$ are Lie irreducible.
\end{lemma}
\begin{proof}
First, the remark at the end of \cite{skinner:hilbert} shows
that $\rho_{\pi,p}$ is irreducible. We will apply \cite[Proposition
1.0.14]{patrikis:tate} to
$\rho_{\pi,p}$ which states that an irreducible Galois representation of a compatible system is of the form $\Ind(\tau\otimes\sigma)$ where $\tau$ is Lie irreducible and $\sigma$ is Artin. Since $\pi$ is not CM, the irreducibility of $\rho_{\pi,p}$ implies that $\rho_{\pi,p}$ is either Lie irreducible or Artin. But $\pi$ is cohomological and $\rho_{\pi,p}|_{G_{F_v}}$ has Hodge--Tate weights $0$ and $1-k_v<0$ and so it cannot be Artin. Finally, the restriction of $\rho_{\pi,p}$ to any open subgroup will contain $\SL(2, \overline{\QQ}_p)$. Since the symmetric power representation of $\SL(2)$ is irreducible, it follows that the restriction of $\Sym^n\rho_{\pi,p}$ to any open subgroup will be irreducible.
\end{proof}

\subsection{Regular submodules}\label{sect:regular submodules}
As previously discussed, $p$-adic $L$-functions are expected to be attached to Galois representations and a choice of regular submodule. In this section, we will describe the possible regular submodules in the case of twists of symmetric powers of Galois representations.

Suppose $F$ is a totally real field and $\pi$ is a cohomological Hilbert modular form of infinity type $(k_1,\ldots, k_d,w)$ over $F$. Let $V_{2n}=\Sym^{2n}\rho_{\pi,p}\otimes\det^{-n}\rho_{\pi,p}$ and 
$V_{2n,v}=V_{2n}|_{G_{F_v}}$. We will classify the regular submodules of $\D_{\st}(V_{2n,v})$ in the case when $p$ splits completely in $F$ and $\pi$ is Iwahori spherical at places $v\mid p$.

Local-global compatibility describes the representation
$\rho_{\pi,p}|_{G_{F_v}}$ completely whenever $v\nmid p$, but
not so when $v\mid p$. We now make explicit the possibilities
for the $p$-adic Galois representation at places $v\mid p$ and in the process we choose a suitable regular submodule of $\D_{\st}(V_{2n,v})$.

Since $p$ splits completely in $F$, $F_v\cong\QQ_p$. The
Galois representation $V=\rho_{\pi,p}|_{G_{F_v}}$ is de Rham
with Hodge--Tate weights $(w-k_v)/2$ and $(w+k_{v}-2)/2$ where we denote by $v$ as
well the unique infinite place in $I_v$. Since $\pi$ is Iwahori
spherical there are two possibilities: either
$\pi_v=\St\otimes\mu$, where $\mu$ is an unramified character, or $\pi_v$ is the unramified principal
series with characters $\mu_1$ and $\mu_2$. 

If $\pi_v=\St\otimes\mu$ then $V$ is semistable but not
crystalline, $k_v$ is even, $\D_{\st}^*(V)=\QQ_p
e_1\oplus \QQ_pe_2$, $\varphi=\begin{pmatrix} \lambda& \\
& p \lambda \end{pmatrix}$ for $\lambda=\mu(\Frob_p)$ with
$v_p(\lambda)=(k_v-2)/2$ and $N=\begin{pmatrix}0 &1\\
0&0 \end{pmatrix}$. The filtration jumps at the Hodge--Tate weights $(w-k_v)/2$ and $(w+k_v-2)/2$ and the proper filtered pieces are given by 
$\QQ_p(e_2-\mathcal{L} e_1)$ for some
$\mathcal{L} \in \QQ_p$, which is the Fontaine--Mazur $\mathcal{L}$-invariant of $V$ (and the Benois $\mathcal{L}$-invariant of $V$). Writing $f_{i}=e_1^{n+i}e_2^{n-i}$, with $n\geq i\geq-n$, for the basis of $\D_{\st}(V_{2n,v})$, the Frobeniuns map $\varphi(f_i)=p^{-i}f_i$ is represented by a diagonal matrix while the monodromy is upper triangular with off-diagonal entries $2n,2n-1,\ldots,1$. The $(\varphi,N)$-stable submodules of $\D_{\st}(V_{2n,v})$ are the spans $\langle f_{n}, f_{n-1}, \ldots, f_i\rangle$. Note that
\[\Fil^0\D_{\dR}^*(V_{2n,v})=\left\langle(e_2-\mathcal{L} e_1)^{2n}, (e_2-\mathcal{L} e_1)^{2n-1}e_1,(e_2-\mathcal{L} e_1)^{2n-1}e_2,\ldots,(e_2-\mathcal{L} e_1)^ne_1e_2^{n-1},(e_2-\mathcal{L} e_1)^ne_2^n\right\rangle
\]
and hence is $(n+1)$-dimensional. Thus, a regular submodule must be $n$-dimensional. The only $n$-dimensional $(\varphi,N)$-stable submodule is $D=\langle f_{n},\ldots,f_{1}\rangle$
and if $\mathcal{L} \neq 0$ (as is expected), then $D$ is regular. In this case, $\D_{\st}(V_{2n,v})^{\varphi=1}=\langle f_0\rangle$ and $D^{\varphi=p^{-1}}=\langle f_1\rangle$. Thus, $D_1=\langle f_n,\dots, f_0\rangle$ (since $N^{-1}(\langle f_1\rangle)=\langle f_0\rangle$) and $D_{-1}=\langle f_n,\dots,f_2\rangle$.

If $\pi_v$ is unramified then $V$ is crystalline. Let
$L/\QQ_p$ be the finite extension generated by the roots
$\alpha=\mu_1(\varpi_v)$ and $\beta=\mu_2(\varpi_v)$ of
$x^2-a_vx+p^{k_v-1}$ with $v_p(\alpha)\leq v_p(\beta)$. Then after base change to $L$,
$\D_{\cris}(V)^*=Le_1 \oplus Le_2$. There are now two possibilities. Either the local representation $V$ splits as $\mu\oplus\mu^{-1}\c^{k-1}$, where
$\mu$ is unramified, or
$V$ is indecomposable. The former case is expected to occur only
when $\pi$ is CM.

If $V$ splits then $\pi_v$ is ordinary $\D_{\cris}^*(V_{2n,v})=\bigoplus_{i=-n}^nLt^{i(k_v-1)}$ (basis $f_i = e_1^{n+i}e_2^{n-i}$), $\varphi$ has
eigenvalues $\alpha^{2n}p^{n(k_v-1)}, \ldots, \alpha^{-2n}p^{-n(k_v-1)}$
where $\mu(\Frob_p)=\alpha$. The de Rham tangent space is
$\D_{\cris}^*(V_{2n})/\Fil^0\D_{\cris}^*(V_{2n})=L f_{1}\oplus
\cdots\oplus L f_{n}$ and so the only regular
subspace is $D=L f_{1}\oplus \cdots\oplus
L f_{n}$. In this case the filtration on $D$ is
given by $D_{-1}=D_0=D$ and $D_1=D\oplus Lf_0$.

If $V$ is not split, we will assume that $\alpha/\beta\not\in\mu_\infty$. Then, we choose $e_1$ and $e_2$ to be eigenvectors of $\varphi$ so that $\D_{\cris}^*(V)=L e_1\oplus Le_2$ with
$\varphi=\begin{pmatrix} \alpha &\\ &\beta \end{pmatrix}$. Moreover, we can scale $e_1$ and $e_2$ so that the one-dimensional filtered piece is $\langle e_1+e_2\rangle$. We remark that $V$ is
reducible if and only if it is
ordinary. Again taking $f_i = e_1^{n+i}e_2^{n-i}$ as a basis, the
Frobenius on $\D_{\cris}^*(V_{2n,v})$ is diagonal with $\varphi(f_i)=(\alpha/\beta)^i$. The de Rham tangent space is
generated by homogeneous polynomials in $e_1$ and $e_2$ which
are not divisible by $(e_1+e_2)^n$. Thus, any choice of $n$ basis
vectors in $f_{n},\ldots, f_{-n}$ will generate a regular
submodule. The assumption that $\alpha/\beta\not\in\mu_\infty$ implies that $\phi(f_i)=f_i$ only for $i=0$. Since $\alpha$ and $\beta$ are Weil numbers of the same complex absolute value, the eigenvalue $p^{-1}$ does not show up. Therefore, no matter what choice of $D$ we take, we have $D_{-1}=D_0=D$. We choose the regular submodule $D=\langle f_n,f_{n-1},\ldots, f_1\rangle$. Since $f_0$ is not among the chosen basis vectors, $D_1 = D\oplus\langle f_0\rangle=\langle{f_n,f_{n-1},\dots,f_0}\rangle$.

In all cases, we have $D_0=D=\langle{f_n,f_{n-1},\dots,f_1}\rangle$ and $D\oplus\langle f_0\rangle=\langle{f_n,f_{n-1},\dots,f_0}\rangle$. Therefore,
\[\gr_1\Drig(V_{2n,v})\cong \mathcal{R}_L.\]

\subsection{Automorphy of symmetric powers}
Let $\pi$ be a cuspidal automorphic representation of
$\GL(2,\AA_F)$ over a totally real field $F$ as in \S \ref{sect:hmf def}. We say that $\Sym^n\pi$ is (cuspidal) automorphic on $\GL(n+1,\AA_{F'})$ for a number field $F'/F$ if there exists a (cuspidal) automorphic representation $\Pi_n$ of $\GL(n+1,\AA_{F'})$ such that $L(\Pi_n,\std,s)=L(\BC_{F'/F}(\pi),\Sym^n,s)$. A cuspidal representation $\sigma$ of $\GL(2, \AA_{F})$ is said to be dihedral if it is isomorphic to its twist by a quadratic character, in which case there exists a CM extension $E/F$ and a character $\psi:\AA_E^\times/E^\times\to \CC^\times$ such that $\sigma\cong\AI_{E/F}\psi$; we say that $\sigma$ is tetrahedral (resp.~octahedral) if $\Sym^2\sigma$ (resp.~$\Sym^3\sigma$) is cuspidal and is isomorphic to its twist by a cubic (resp.~quadratic) character; we say that $\sigma$ is solvable polyhedral if $\sigma$ is dihedral, tetrahedral or octahedral; we say that $\sigma$ is icosahedral if $\Sym^5\sigma\cong\Ad(\tau)\sqtimes\sigma\otimes\omega_\sigma^2$ for some cuspidal representation $\tau$ of $\GL(2, \AA_F)$ and if $\sigma$ is not solvable polyhedral. The cuspidal representation $\pi$, being associated with a regular Hilbert modular form, is either dihedral (if the Hilbert modular form is CM) or not polyhedral.
\begin{theorem}\label{t:automorphy of sym}Suppose $\pi$ is as in \S \ref{sect:hmf def}. Then
\begin{enumerate}
\item $\Sym^m\pi$ is automorphic for $m=2$ (\cite[Theorem 9.3]{gelbart-jacquet}), $m=3$ (\cite[Corollary 1.6]{kim-shahidi:sym3} and \cite[Theorem 2.2.2]{kim-shahidi:cuspidality-sym}) and $m=4$ (\cite[Theorem B]{kim:sym4} and \cite[3.3.7]{kim-shahidi:cuspidality-sym}); it is cuspidal unless $\pi$ is CM.

\item Suppose $\pi$ is not CM. If $\Sym^5\pi$ is automorphic
then it is cuspidal. If $m\geq 6$ suppose either that
$\Sym^k\pi$ is automorphic for $k\leq 2m$ or that $\pi\sqtimes \tau$ is automorphic for any cuspidal representation $\tau$ of $\GL(r, \AA_{F})$ where $r\leq \lfloor m/2+1\rfloor$. Then $\Sym^m\pi$ is cuspidal (\cite[Theorem A']{ramakrishnan:remarks-sym}).
\end{enumerate}
\end{theorem}

\begin{remark}
Results on the automorphy of $\Sym^n\pi$ for small values of $n$ have been obtained by Clozel and Thorne assuming conjectures on level raising and automorphy of tensor products. Potential automorphy results follow from automorphy lifting theorems.
\end{remark}

\subsection{Functorial transfers to unitary and symplectic
  groups}
To compute $\mathcal{L}$-invariants, we use $p$-adic families of Galois representations. However, since $\GL(n)$ for $n>2$ has no associated Shimura variety, to
construct $p$-adic families of automorphic representations, we transfer to unitary or symplectic groups.

We begin with $\GSp(2n)$. Let $\omega_0$ be the $n\times n$ antidiagonal matrix with 1-s
on the antidiagonal and let $\GSp(2n)$ be the reductive group of
matrices $X\in \GL(2n)$ such that $X^TJX=\nu(X)J$ where
$J=\begin{pmatrix} & \omega_0\\-\omega_0&\end{pmatrix}$ and
$\nu(X)\in \mathbb{G}_m$ is the multiplier character. The
diagonal maximal torus $T$ consists of matrices $t(x_1,\ldots,
x_g, z)$ with $(x_1,\ldots, x_g, zx_g^{-1},\ldots, zx_1^{-1})$
on the diagonal. The Weyl group $W$ of $\GSp(2n)$ is $S_n\rtimes
(\ZZ/2 \ZZ)^n$ and thus any $w\in W$ can be
written as a pair $w=(\nu, \varepsilon)$ where $\nu\in S_n$ is a
permutation and $\varepsilon:\{1,\ldots, n\}\to \{-1,1\}$ is a
function. The Weyl group acts by conjugation on $T$ and $(\nu,
\varepsilon)$ takes $t(x_1, \ldots, x_n, z)$ to $t(x'_{\nu(1)},
\ldots, x'_{\nu(n)}, z)$ where $x'_i = x_i$ if
$\varepsilon(i)=1$ and $x'_i = zx_i^{-1}$ if
$\varepsilon(i)=-1$. We choose as a basis for $X^\bullet(T)$ the characters $e_i(t(x_i;z))=x_iz^{-1/2}$ and $e_0((x_i;z))=z^{1/2}$. 
Let $B$ be the Borel subgroup of upper triangular matrices, which corresponds to the choice of simple roots $e_i-e_{i+1}$ for $i<n$ and $2e_n$. Half the sum of positive roots is then $\rho = \sum_{i=1}^n(n+1-i)e_i$. The compact roots are $\pm(e_i-e_j)$ and so half the sum of compact roots is $\rho_c = \sum_{i=1}^n(n-2i+1)/2 e_i$. This gives the relationship between Harish-Chandra and Blattner parameters for $\GSp(4)$ as $\lambda_{\operatorname{HC}}+\sum_{i=1}^n ie_i = \lambda_{\operatorname{Blattner}}$.
 
 \begin{theorem}[Ramakrishnan--Shahidi]\label{t:GL(4) to GSp(4)}
Let $\pi$ be as in \S \ref{sect:hmf def} of infinity type $(k_1,\ldots,k_d,w)$. If $\pi$ is not CM there exists a cuspidal automorphic representation $\Pi$ of $\GSp(4,\AA_{F})$ which is a strong lift of $\Sym^3\pi$ from $\GL(4,\AA_F)$ such that for every infinite place $\tau$, $\Pi_\tau$ is the holomorphic discrete series with Harish-Chandra parameters $(2(k_\tau-1),k_\tau-1;-3w/2)$. 
\end{theorem}
\begin{proof}
When $F=\QQ$ and the weight is even this is \cite[Theorem A$^\prime$]{ramakrishnan-shahidi}. Ramakrishnan remarks that the proof of this result should also work for totally real fields. We give a proof using Arthur's results on the discrete spectrum of symplectic groups.

Let $\sigma$ be the cuspidal representation of $\GL(4,
\AA_F)$ whose $L$-function coincides with that of
$\Sym^3\pi$. Then
$L(\wedge^2\sigma\otimes\det^3\rho_{\pi,p},s)=L(\wedge^2\Sym^3\rho_{\pi,p}\otimes\det^3\rho_{\pi,p},s)=\zeta(s)L(\Sym^4\rho_{\pi,p}\otimes\det^2\rho_{\pi,p},s)$
has a pole at $s=1$ (for example because $\Sym^4\pi$ is cuspidal
automorphic) so \cite[Theorem
12.1]{gan-takeda:local-langlands-GSp(4)} shows that there exists
a globally generic (i.e., having a global Whittaker model)
cuspidal representation $\Pi^g$ of $\GSp(4, \AA_F)$
strongly equivalent to $\sigma$, and thus to $\Sym^3\pi$. For
every infinite place $\tau$ the representation $\Pi^g_\tau$ will
be a generic (nonholomorphic) discrete series
representation. Let $\psi$ be the global $A$-parameter attached
to $\Pi^g$; since $\Pi^g$ is globally generic with cuspidal lift
to $\GL(4)$, $\psi$ will be simple generic and therefore the
$A$-parameter $\psi$ is in fact an $L$-parameter and the
component group of $\psi$ is trivial. Let
$\Pi=(\otimes_{\tau\mid\infty}\Pi^h_\tau)\otimes(\otimes_{v\nmid\infty}\Pi^g_v)$
be the representation of $\GSp(4, \AA_F)$ obtained using
the holomorphic discrete series $\Pi^h_\tau$ in the same
archimedean local $L$-packet as $\Pi^g_\tau$. Then Arthur's
description of the discrete automorphic spectrum for $\GSp(4)$
implies that $\Pi$ is an automorphic representation (for
convenience, see \cite[Theorem 2.2]{mok:siegel-hilbert}). Since
the representations $\Pi_\tau$ at infinite places $\tau$ are
discrete series, they are also tempered and so \cite[Theorem
4.3]{wallach:constant} implies that $\Pi$ will also be
cuspidal. By construction, $\Pi$ will be strongly equivalent to
$\Sym^3\pi$.

It remains to compute the Harish-Chandra parameter of
$\Pi_\tau$. Let $\phi_{\pi_\tau}$ and $\phi_{\Pi_\tau}$ be the
Langlands parameters of $\pi_\tau$ respectively $\Pi_\tau$. Then
$\phi_{\Pi_\tau}=\Sym^3\phi_{\pi_\tau}$ and so
\[\phi_{\Pi_\tau}(z)=|z|^{-3w}\begin{pmatrix}
(z/\overline{z})^{(3(k_\tau-1))/2}&&&\\ &
(z/\overline{z})^{(k_\tau-1)/2}&&\\ &&
(z/\overline{z})^{-(k_\tau-1)/2}&\\ &&&
(z/\overline{z})^{-(3(k_\tau-1))/2}\end{pmatrix}\]
The recipe from \cite[\S 2.1.2]{sorensen:gsp4} shows that the $L$-packet defined by $\phi_{\Pi_\tau}$ consists of the holomorphic and generic discrete series with Harish-Chandra parameters $(2(k_\tau-1),k_\tau-1;-3w/2)$. 

\end{proof}

For higher $n$ one does not yet have transfers from $\GL(n)$ to
similitude symplectic groups, although one may first transfer from
$\GL(2n+1)$ to $\Sp(2n)$ and then lift to $\GSp(2n)$.
\begin{theorem}\label{t:GL(2n+1) to Sp(2n)}Let $F$ be a totally real field and $\pi$ be a regular algebraic cuspidal automorphic self-dual
representation of $\GL(2n+1, \AA_F)$ with trivial central character. Then there exists a cusidal automorphic representation $\overline{\sigma}$ of $\Sp(2n, \AA_F)$ which is a weak functorial transfer of $\pi$ such that $\overline{\sigma}$ is a holomorphic discrete series at infinite places. If, moreover, $\pi$ is the symmetric $2n$-th power of a cohomological Hilbert modular form then there exists a cuspidal representation 
$\sigma$ of $\GSp(2n,\AA_F)$ which is a holomorphic discrete series at infinity and such that any irreducible component of the restriction $\sigma|_{\Sp(2n,\AA_F)}$ is in the same global $L$-packet at $\overline{\sigma}$.
\end{theorem}
\begin{proof}
We will use \cite[Theorem
3.1]{ginzburg-rallis-soudry:descent} to produce an irreducible
cuspidal globally generic functorial transfer $\tau$ of a cuspidal self-dual representation $\pi$ of $\GL(2n+1, \AA_F)$ to
$\Sp(2n, \AA_F)$, and all citations in this paragraph are
from \cite{ginzburg-rallis-soudry:descent}. In the notation of
\S 2.3 the group $H$ is taken to be the metaplectic double cover
$\widetilde{\Sp}(4n+2)$ of $\Sp(4n+2)$ (this is case 10 from
(3.43)) in which case, in the notation of \S 2.4, $L(\pi,
\alpha^{(1)},s)=1$ and $L(\pi,\alpha^{(2)}, s)=L(\pi,\Sym^2,s)$
which has a pole at $s=1$ since $\pi\cong \pi^\vee$. Let $E_\pi$
be the irreducible representation of $H(\AA_F)$ of
Theorem 2.1 generated by the residues of certain Eisenstein
series attached to the representation $\pi$ thought of as a
representation of the maximal parabolic of $H$. One finds $\tau$
as an irreducible summand of the automorphic representation
generated by the Fourier--Jacobi coefficients of $E_\pi$ where,
in the notation of \S 3.6, one takes $\gamma=1$. The fact that
$\tau$ is cuspidal and globally generic and a strong transfer of
$\pi$ then follows from Theorem 3.1.

Next, Arthur's global classification of the discrete spectrum of
$\Sp(2n)$ (\cite[Theorem 1.5.2]{arthur:endoscopic-classification}) implies that there exists a cuspidal automorphic representation $\tau'$ which is isomorphic to $\tau$ at all finite places and the representation $\tau'_\infty$ is a holomorphic discrete series in the same $L$-packet at $\tau_\infty$.

Finally, Proposition 12.2.2, Corollary 12.2.4 and Proposition 12.3.3 (the fact that $\pi$ is the symmetric $2n$-th power of a {\it cohomological} Hilbert modular form implies that hypothesis (2) of this proposition is satisfied) of \cite{patrikis:tate} imply the existence of a regular algebraic cuspidal automorphic representation $\sigma$ of $\GSp(2n, \AA_F)$ such that if $v$ is either an infinite place or a finite place such that $\sigma_v$ is unramified then $\sigma_v|_{\Sp(2n, F_v)}$ contains $\tau'_v$. Moreover, the discrete series $\sigma_\infty$ is holomorphic or else its restriction to $\Sp(2n, \RR)$ would not contain the holomorphic discrete series $\tau'_\infty$.
\end{proof}

One reason to seek a formula for symmetric power $\mathcal{L}$-invariants in terms of $p$-adic families on a certain reductive group is that it might yield a proof of a trivial zero conjecture for symmetric powers of Hilbert modular forms following the template of the Greenberg--Stevens proof of the Mazur--Tate--Teitelbaum conjecture. This method for proving such conjectures requires the existence of $p$-adic $L$-functions for these $p$-adic families. Whereas there has been little progress towards such $p$-adic families of $p$-adic $L$-functions on symplectic groups in general, the work of Eischen--Harris--Li--Skinner is expected to yield such $p$-adic $L$-functions in the case of unitary groups. We will therefore present a computation (however, under some restrictions) of the symmetric power $\mathcal{L}$-invariants using unitary groups.

For a CM extension $E$ of a totally real
field $F$, let $U_n$ be the unitary group defined in
\cite[Definition 6.2.2]{bellaiche-chenevier:selmer} which is
definite at every finite place and if $n\neq 2\mod{4}$ is
quasi-split at every nonsplit place of $F$. We denote by $\BC$
the local base change (see, for example, \cite[\S 2.3]{shin:clay09}). Assume the following conjecture on strong base change that would follow from stabilization of the trace formula.

\begin{conjecture}\label{t:GL(n) to U_n}
Let $\pi$ be a regular algebraic conjugate self-dual cuspidal
automorphic representation of $\GL(n, \AA_E)$ such that $\pi_v$ is the base change from $U$ of square integrable representations at ramified places of $E$ and is either unramified or the base change from $U$ of square integrable represenations at places of $F$ which are inert in $E$. Suppose, moreover, that for at least one inert prime the local representation is not unramified. Then there exists a (necessarily cuspidal) automorphic representation $\Pi$ of $U_n(\AA_F)$ such that $\Pi_w=\BC(\pi_v)$ for all places $v$.
\end{conjecture}
We will use this conjecture to
transfer symmetric powers of Hilbert modular forms to unitary
groups.

\begin{proposition}\label{p:sym hmf to unitary}
Assume Conjecture \ref{t:GL(n) to U_n}. Let  $\pi$ be an Iwahori spherical cohomological non-CM Hilbert modular form over a number field $F$ in which $p$ splits completely. Suppose that there exist places $w_1$ and $w_2$ not above $p$ with the property that $\pi_{w}$ is special for $w\in\{w_1,w_2\}$ and suppose that $\Sym^n\pi$ is cuspidal automorphic over $F$. Then there exists a CM extension $E/F$ in which $p$ splits completely and a Hecke character $\psi$ of $E$ such that $\psi\otimes\BC_{E/F}\Sym^n\pi$ is the base change of a cuspidal automorphic representation of $U_n(\AA_F)$. 
\end{proposition}
\begin{proof}
There exists $\alpha\in \{\varpi_{v_1}, \varpi_{v_2}, \varpi_{v_1}\varpi_{v_2}\}$ having a square root in
$\mathbb{F}_{p^{[F:\QQ]}}$. Then $E=F(\sqrt{\alpha})$ is ramified over
$F$ only at $w_1$ or $w_2$ (or both) and $p$ splits completely in $E$.

Suppose that $\Pi=\BC_{E/F}\Sym^n\pi$ is cuspidal
automorphic. Then $\Pi^{c\vee}\cong \Pi\otimes
\BC_{E/F}\omega_\pi^{-n}$. Choose a Hecke character
$\psi:E^\times\backslash \AA_E^\times\to
\CC^\times$ such that $\psi|_{\AA_F^\times} =
\omega_\pi^{-n}$. Then $(\psi\otimes\Pi)^{c\vee}\cong
\psi^{c\vee}\omega_\pi^{-n}\otimes\Pi\cong
\psi\otimes\Pi$. Since $\Sym^n\pi$ is Iwahori spherical it
follows that $\Pi$ is Iwahori spherical. Moreover, if $v$ is a
ramified place of $E/F$, $\pi_v$ is special by assumption
and thus $\Pi_v$ is the base change from $U$ of a square integrable representation. Finally, at every finite
place $\Pi_v$ is either unramified or special. Therefore, the
hypotheses of Conjecture \ref{t:GL(n) to U_n} are satisfied
and the conclusion follows.

It remains to show that $\Pi\cong\Sym^n\BC_{E/F}\pi$ is cuspidal
automorphic. Writing $\pi_E=\BC_{E/F}\pi$ we note that $\pi_E$
is cuspidal since $\pi$ is not CM. Therefore one may apply
\cite[Theorem A']{ramakrishnan:remarks-sym} to study the
cuspidality of $\Sym^n\pi_E$. If $\pi_E$ is dihedral then
$\pi_E\cong\Ind_{L}^E\eta$ for some quadratic extension $L/E$
contradicting the fact that $\rho_{\pi,p}$ is Lie irreducible.

If $\pi_E$ is tetrahedral then $\Sym^3\pi$ is cuspidal but
$\BC_E\Sym^3\pi$ is not which implies that
$\Sym^3\pi\cong\Ind_E^F\tau$ for an automorphic representation
$\tau$ of $\GL(2, \AA_F)$. It suffices to show that there exists a Galois representation attached to $\tau$ since then $\Sym^3\pi$ is not Lie irreducible contradicting Lemma \ref{l:hmf lie irreducible}. Since a real place $v$ of $F$ splits completely in $E$, the $L$-parameter of $(\Ind_E^F\tau)_v$ is the direct sum of the $L$-parameters of the local components of $\tau$ at the complex places over $v$. Therefore $\tau$ is regular algebraic, necessarily cuspidal, representation of $\GL(2, \AA_E)$. Galois representations have been attached to such cuspidal representations by Harris--Lan--Taylor--Thorne, although one does not need such a general result; indeed, by twisting one may guarantee that the central character of $\tau$ is trivial in which case the Galois representation has been constructed by Mok (\cite{mok:siegel-hilbert}).

If $\pi_E$ is icosahedral, i.e., $\Sym^6\pi_E$ is not cuspidal but all lower symmetric powers are cuspidal then $\Sym^6\pi_E=\eta\boxplus\eta'$ where $\eta$ is a cuspidal representation of $\GL(3, \AA_F)$ and $\eta'$ is a cuspidal representation of $\GL(4, \AA_F)$ (cf. \cite[\S 4]{ramakrishnan:remarks-sym}). Lemmas 4.10 and 4.18 of \cite{ramakrishnan:remarks-sym} carry over to this setting and show that $\eta$ and $\eta'$ are regular, algebraic conjugate self-dual representations and thus that $\Sym^6\rho_{\pi,p}|_{G_E}$ is decomposable, which contradics Lemma \ref{l:hmf lie irreducible}.

Finally, if $\pi_E$ is octahedral then $\Sym^4\pi_E$ is not cuspidal. As in the case of $\Sym^6$ we may write $\Sym^4\pi_E=\eta\boxplus\eta'$ where $\eta$ is a regular algebraic cuspidal conjugate self-dual representation of $\GL(2, \AA_E)$ and $\eta'$ is a regular algebraic cuspidal conjugate self-dual representation of $\GL(3, \AA_E)$ yielding a contradiction as above.

\end{proof}

\section{$p$-adic families and Galois representations}\label{sect:eigs}
As previously mentioned, we will compute explicitly the
arithmetic $\mathcal{L}$-invariants attached to $V_{2n}$ using
triangulations of $(\varphi,\Gamma)$-modules attached to
analytic Galois representations on eigenvarieties. Which
eigenvariety we choose will be dictated by the requirement that
the formula from Lemma \ref{l:rank 1 L invariant} needs to
make sense. This translates into a lower bound for the rank of
the reductive group whose eigenvariety we will use. In this
section we make explicit the eigenvarieties under consideration
and the analytic Hecke eigenvalues whose derivatives will
control the arithmetic $\mathcal{L}$-invariant.

\subsection{Urban's eigenvarieties}
We recall Urban's construction of eigenvarieties from
\cite[Theorem 5.4.4]{urban:eigenvarieties}, to which we refer
for details. Let $G$ be a split reductive group over $\QQ$ such
that $G(\RR)$ has discrete series representations. (This
condition is satisfied by the restriction to $\QQ$ of
$\GL(2)$ and $\GSp(2n)$ over totally real fields and compact
unitary groups attached to CM extensions.) We denote by $T$ a maximal (split) torus, by $B$ a Borel subgroup containing $T$ and by $B^-$ the opposite Borel, obtained as $B^-=w_BBw_B^{-1}$ for some $w_B\in W_{G,T}=N_G(T)/T$. For a character $\lambda\in X^\bullet(T)$ let $V_\lambda=\Ind_{B^-}^G \lambda$ induced from the opposite Borel, the irreducible algebraic representation of $G$ of highest weight $\lambda$ with respect to the set of positive roots defined by the Borel $B$. Its dual is $V_\lambda^\vee\cong \Ind_B^G \left((-w_B)\lambda\right)$.

Let $\pi$ be a
cuspidal automorphic representation of
$G(\AA_{\QQ})$ and assume that it has regular cohomological weight $\lambda\in X^\bullet(T)$, i.e.\ that $\pi$ can be realized in the cohomology of the local system $V_\lambda^\vee$; equivalently $\pi_\infty$ has central character equal to the central character of $\lambda$ and has a twist which is the discrete series representation of $G(\RR)$ of Harish-Chandra parameter $\lambda$ (cf. \cite[\S 1.3.4]{urban:eigenvarieties}). For example, if $f$ is a classical modular form of weight $k\geq 2$ (as before $w=2-k$) then the cohomological weight of the associated automorphic representation is $(k-2, 0)\in X^\bullet(\mathbb{G}_m^2)$ as the Eichler--Shimura isomorphism implies that $f$ appears in the cohomology of $(\Sym^{k-2})^\vee$.

Let $\AA_{\QQ,f}^{(p)}$ be the finite adeles away from $p$ and let $K^p\subset G(\AA_{\QQ,f}^{(p)})$ be a
compact open subgroup. Let $I_m'$ be a pro-$p$-Iwahori subgroup of $G(\Qp)$ such that
$\pi_f^{K^pI_m'}\neq 0$. We denote by $\mathcal{H}^p$ the Hecke
algebra of $K^p$, $\mathcal{U}_p$ the Hecke algebra of the Iwahori subgroup $I_m=\{g\in G(\Zp):g\mod{p}\in B(\Zp)\}$ and
$\mathcal{H}=\mathcal{H}^p\otimes \mathcal{U}_p$. A
$p$\textit{-stabilization} $\nu$ of $\pi$ is an irreducible constituent of $\pi_f^{K^pI_m'}\otimes \varepsilon^{-1}$ as a $\mathcal{H}$-representation, where $\varepsilon$ is a character of $I_m/I_m'\cong T(\ZZ/p^m \ZZ)$ acting on $\pi_f^{K^pI_m'}$ (cf. \cite[p. 1689]{urban:eigenvarieties}). For $\nu$ finite slope of weight $\lambda$, Urban
rescales  by
$|\lambda(t)|_p^{-1}$ the eigenvalue of the Hecke operator $U_t=I_mtI_m\in
\mathcal{U}_p$, where $t\in T(\QQ_p)$ is such that $|\alpha(t)|_p\leq1$ for all positive roots $\alpha$ (\textit{finite slope} means the $U_t$ act invertibly; cf. \cite[p.~1689--1690]{urban:eigenvarieties}). This rescaled eigenvalue is denoted
$\theta(U_t)$. The $p$-stabilization $\nu$ of $\pi$ is said to have
\textit{non-critical slope} if there is a Hecke operator $U_t$, with $|\alpha(t)|_p<1$ for all positive roots $\alpha$, such that for every simple root $\alpha$
\[v_p(\theta(U_t))<(\lambda(\alpha^\vee)+1)v_p(\alpha(t)).\]

\begin{theorem}[{\cite[Theorem 5.4.4]{urban:eigenvarieties}}]\label{t:urban}
Suppose $\nu$ has non-critical slope. Then there exists a rigid analytic ``weight space'' $\mathcal{W}$ of the weight $\lambda$ of $\pi$, a rigid analytic variety $\mathcal{E}$, a generically finite flat morphism $w:\mathcal{E}\to \mathcal{W}$ and a homomorphism $\theta:\mathcal{H}\to \mathcal{O}_{\mathcal{E}}$ such that:
\begin{enumerate}
\item there exists a dense set of points $\Sigma\subset \mathcal{E}$ and 
\item for every point $z\in \Sigma$ there exists an irreducible
cuspidal representation $\pi_z$ of $G(\AA_{\QQ})$ of
weight $w(z)$, with Iwahori spherical local representations at places over $p$, and a $p$-stabilization $\nu_z$ such that $\theta(T)$ is the normalized eigenvalue of $T\in \mathcal{H}$ on $\nu_{z}$ and 
\item there exists $z_0\in \Sigma$ such that $\pi_{z_0}=\pi$ and $\nu_{z_0}=\nu$.
\end{enumerate}

\end{theorem}

Suppose $F$ is a totally real field in which $p$ splits completely and $G/F$ is a split
reductive group over $F$. Let $G'=\Res_{F/\QQ}G$ be the
Weil restriction of scalars. Then a cuspidal automorphic
representation $\pi$ of $G(\AA_F)$ can be thought of as a
cuspidal representation of $G'(\AA_{\QQ})$ and we
may apply Urban's construction above. Let $T$ be a maximal torus
of $G$ and $T'=\Res_{F/\QQ}T$. Then $X^\bullet(T')\cong
X^\bullet(T)\cong\oplus_{G_{F/\QQ}}X^\bullet(T_{/\QQ_p})$
gives a one-to-one correspondence between the weights of
automorphic representations of $G(\AA_F)$ and tuples
$(w_1,\ldots)$ where $w_i\in X^\bullet(T_{/\QQ_p})$. For
each place $v\mid p$, given $t\in T(\QQ_p)$ one gets a
Hecke operator $U_t$ acting on $\pi_v$ and the refinement
$\nu_v$ with eigenvalue $\alpha_v$. Then the eigenvalue of $U_t$
on the refinement $\nu$ of $\pi$ thought of as a representation
of $G'(\AA_{\QQ})$ is $\prod_v \alpha_v$ and the
slope of the renormalized eigenvalue is
\[v_p(\theta(U_t))=\sum_v v_p(\lambda_v(t)) + \sum_v v_p(\alpha_v)=\sum_v v_p(\theta_v(U_t))\]
where $\theta_v(U_t)$ is the eigenvalue of $U_t$ on $\pi_v$ renormalized by the weight $\lambda_v$. If $\Phi$, $\Phi^+$ and $\Psi$ are the set of roots, positive roots, and simple roots of $G_{/\QQ_p}$ with respect to a Borel, then the set of roots, positive roots, and simple roots of $G$ are $\prod_i \Phi$, $\prod_i \Phi^+$, and $\{0\oplus\cdots\oplus \alpha\oplus\cdots\oplus 0|\alpha\in \Psi\}$.

In the next sections, we will make explicit first what
non-critical slope means in the case of the groups under
consideration, and second, the relationship between the
renormalized action of Hecke operators on $p$-stabilizations and their
action on smooth representations of $p$-adic groups. The latter
will allow us to study global triangulations of Galois
representations in terms of analytic Hecke eigenvalues.

\subsection{Eigenvarieties for Hilbert modular forms}
Suppose $\pi$ is a cohomological Hilbert modular form of infinity type
$(k_1,\ldots, k_d, w)$ over a totally real field $F$ of degree
$d$ in which $p$ splits completely. The representation $\pi$ has cohomological weight $\bigoplus_i ((-w+k_i-2)/2, (-w-k_i+2)/2)$. For each place $v\mid p$, let
$\alpha_v$ be an eigenvalue of $U_p$ (corresponding to $t=\begin{pmatrix} 1&\\ &p^{-1}\end{pmatrix}$) acting on a refinement $\nu$ of
$\pi$ in which case $\theta(U_p)=p^{(w+k_i-2)/2}\alpha_v$. Then
$v_p(\theta(U_p))=\sum_i ((w+k_i-2)/2+v_p(\alpha_v))$. The simple roots of $\GL(2)_{/F}$ are of the form $(1,-1)_i$ where $i$ is an infinite place. As $((-w+k-2)/2,(-w-k+2)/2)\cdot(1,-1) = k-2$, the refinement $\nu$ has
noncritical slope if and only if
\[\sum_i ((w+k_i-2)/2+v_p(\alpha_v))<k_i-1\]
for every $i$, which is equivalent to
\[\sum_i ((w+k_i-2)/2+v_p(\alpha_v))<\min(k_i)-1\]
We remark that for modular forms of weight $k$ (with $w=2-k$) this is the usual definition of noncritical slope.

\begin{lemma}\label{l:hmf eigenvariety}
Let $\pi$ be a cohomological Hilbert modular form, $\nu$ a refinement of $\pi$ of noncritical slope as above, and
$\mathcal{E}\to \mathcal{W}$ Urban's eigenvariety around
$(\pi,\nu)$. Shrinking $\mathcal{W}$ if necessary, there exists an analytic
Galois representation $\rho_{\mathcal{E}}:G_F\to \GL(2,
\mathcal{O}_{\mathcal{E}})$ such that for $z\in \Sigma$, $z\circ \rho_{\mathcal{E}}=\rho_{\pi_z}$. If $\pi$ has Iwahori level at $v\mid p$ then $\rho_{\mathcal{E}}$ admits a refinement in the sense of \S \ref{sect:refined galois}.
\end{lemma}
\begin{proof}
The existence of $\rho_{\mathcal{E}}$ is done by a standard
argument, but we reproduce it here for convenience. If $v\nmid
p$ is a place such that $\pi_w$ is unramified and $z\in \Sigma$
then $z(\theta(T_v))$ is equal to the eigenvalue of $T_v$ on
$\pi_z$. Local-global compatibility implies that
$\Tr\rho_{\pi_z}(\Frob_v)=z\circ\theta(T_v)$ and so
$\theta(T_v)$ is an analytic function $T(\Frob_v)$ specializing
at points $z\in\Sigma$ to $\Tr\rho_{\pi_z}(\Frob_v)$. As
$\{\Frob_v|v\nmid Np\}$ is dense in $G_F$ we may define by
continuity $T(g) = \lim T(\Frob_v)$ where $\Frob_v\to g$. The
function $T:G_F\to \mathcal{O}_{\mathcal{E}}$ is an analytic
two-dimensional pseudorepresentation. Since $\pi$ is cuspidal,
$\rho_\pi$ is irreducible
(cf. \cite[p. 256]{skinner:hilbert}). This implies that
shrinking $\mathcal{W}$, $T$ is the trace of an analytic Galois
representation $\rho_{\mathcal{E}}$
(cf. \cite[4.2.6]{jorza:thesis}).

If $\pi$ has Iwahori level at $v\mid p$ then $\pi_z$ will also have Iwahori level at $v\mid p$ for $z\in \Sigma$. Thus $\rho_{z,v}$ for $v\mid p$ is either unramified or Steinberg. If Steinberg then the unique refinement of $\pi_{z,v}$ has slope $k_v-2+w/2$. Shrinking the neighborhood $\mathcal{W}$ we may guarantee that the slope is constant. The classical points on $\mathcal{E}$ are dense and the ones satisfying $k_v-2+w/2$ constant lie in a subvariety and so there is a dense set $\Sigma$ of points on $\mathcal{E}$ corresponding to cohomological regular Hilbert modular forms.

Suppose $z\in \Sigma$ in which case $\rho_{\pi_z,p}|_{G_{F_v}}$ is crystalline for $v\mid p$. By local-global compatibility the eigenvalues of $\varphi_{\cris}$ acting on $\D_{\cris}(\rho_{\pi_z,p}|_{G_{F_v}}$ are $\alpha_vp^{-1/2}$ and $\beta_vp^{-1/2}$ ($F_v\cong \QQ_p$ so $\varpi_v$ can be chosen to be $p$) and by choice of central character for $\pi$ we deduce that $\alpha_v \beta_v=p^{-w}$. Writing $a_v(z) = z(\theta(U_p))$ we get an analytic function on $\mathcal{E}$ such that for $z\in \Sigma$, the eigenvalues of $\varphi_{\cris}$ are $\alpha_v p^{-1/2} = a_v(z) p^{(-w-k_v+1)/2}$ and $\beta_v p^{-1/2}=p^{-w-1/2}\alpha_v^{-1} = a_v(z)^{-1} p^{(-w+k_v-3)/2}$. Let $\kappa_{v,1}(z) = (w-k_v)/2$ and $\kappa_{v,2}(z) = (w+k_v-2)/2$; let $F_1(z) = a_v(z)^{-1}p^{3/2}$ and $F_2(z) = a_v(z)p^{1/2}$. Then for $z\in \Sigma$, $\kappa_{1, v}(z)<\kappa_{2,v}(z)$ are the Hodge--Tate weights of the crystalline representation $\rho_{\pi_z,p}|_{G_{F_v}}$ and the eigenvalues of $\varphi_{\cris}$ are $F_1(z)p^{\kappa_1(z)}$ and $F_2(z)p^{\kappa_2(z)}$ which implies that $\rho_{\mathcal{E}}|_{G_{F_v}}$ admits a refinement as desired. 
\end{proof}

\begin{corollary}\label{c:hmf triangulation}
Let $\pi$ be as in Lemma \ref{l:hmf eigenvariety}. For $v\mid
p$, assume that $\pi_v$ is Iwahori spherical; if $\pi_v$ is an
unramified principal series assume that $\alpha_v\neq\beta_v$. Then there exists a global triangulation
\[\mathscr{D}_{\rig}^\dagger(\rho_{\mathcal{E},
  p}|_{G_{F_v}})\sim \begin{pmatrix} \delta_1&*\\
&\delta_2\end{pmatrix}\]
with $\delta_1(p) = a_v(z)^{-1}p^{3/2}$, $\delta_1(u)=u^{(w-k_v)/2}$, $\delta_2(p)=a_v(z)p^{-1/2}$ and $\delta_2(u) = u^{(w+k_v-2)/2}$.
\end{corollary}
\begin{proof}
We only need to check that the refinement attached to the triangulation of $\rho_{\pi,p}|_{G_{F_v}}$ is noncritical and regular by Theorem \ref{t:triangulation}. Noncriticality is immediate from the requirement that the Hodge--Tate weights be ordered increasingly (cf. \cite[\S 5.2]{liu:triangulation}). Regularity is equivalent to the fact that $\phi_{\cris}$ has distinct eigenvalues, i.e.\ that $\alpha_v\neq \beta_v$.
\end{proof}

\subsection{Eigenvarieties for  $\GSp(2n)$}\label{sect:eigenvarieties for symplectic}
Suppose $F/\QQ$ is a number field in which $p$ splits
completely, $g\geq 2$ an integer, and $\pi$ a cuspidal
automorphic representation of $\GSp(2g, \AA_F)$ such that
for $v\mid\infty$ (writing $v\mid p$ for the associated place
above $p$) the representation $\pi_v$ is the holomorphic
discrete series representation with Langlands parameter, under the $2^g$-dimensional spin representation $\GSpin(2g+1, \CC)\into \GL(2^g,\CC)$, uniquely identified by
\[z\mapsto |z|^{\mu_0}\diag((z/\overline{z})^{1/2\sum \varepsilon(i)(\mu_{v,i}+g+1-i)})\]
where $\varepsilon:\{1,2,\ldots,g\}\to \{-1,1\}$. If $V_\mu$ is
the algebraic representation of $\GSp(2g)$ with highest weight
$\mu_v=(\mu_{v,1},\ldots, \mu_{v,g}; \mu_0)$ then $H^\bullet(\mathfrak{gsp}_{2g}, \SU(g), \pi_v\otimes V_\mu^\vee)\neq 0$ and so $\pi_v$ has cohomological weight $\mu$ in the sense of Urban.

 We denote by $\Iw$ the Iwahori subgroup of $\GSp(2g)$ of matrices which are upper triangular mod $p$. Let $\beta_0 = t(1,\ldots,1;p^{-1})$ and $\beta_j=t(1,\ldots,1,p^{-1},\ldots,p^{-1};p^{-2})$ where $p^{-1}$ appears $j$ times. The
double coset $[\Iw \beta_i\Iw]$ acts on the finite dimensional
vector space $\pi_v^{\Iw}$ for $v\mid p$.

\begin{lemma}\label{l:hecke local gsp}Let $\chi=\chi_1\otimes\cdots\otimes\chi_g\otimes\sigma$ be an unramified character of the diagonal torus of $\GSp(2g, \QQ_p)$ and assume that the unitary induction $\mu=\Ind_B^{\GSp(2g, \QQ_p)}\chi$ is irreducible. Then $\mu^{\Iw}$ is $2^g g!$ dimensional and has a basis in which the operator $U_t=[\Iw t\Iw]$ for any $t=(p^{a_1}, \ldots, p^{a_g}; p^{a_0})$ is upper triangular. The basis is $\{e_w\}$ indexed by the Weyl group elements $w=(\nu,\varepsilon)\in W$ in the Bruhat ordering. The diagonal element of $U_t$ corresponding to $e_w$ is
\[p^{g(g+1)/4a_0-\sum (n+1-j)a_j}\sigma(p)^{a_0}\prod_{j=1}^g \chi_j(p)^{a_{\nu(j)}\textrm{ or }a_0-a_{\nu(j)}}\]
where the exponent depends on whether $\varepsilon(\nu(j))=1$ or $-1$.

In particular, the Hecke operators $U_{p,i}=[\Iw \beta_{g-i} \Iw]$ are upper
triangular whose diagonal elements are
\[\left\{\displaystyle p^{c_{i,\nu,\varepsilon}}\sigma^{-2}(p)\prod_{\nu(j)>i}\chi_j^{-1}(p)\prod_{\nu(j)\leq i,\varepsilon(j)=-1}\chi_j^{-2}(p)\right\}\text{ for }
1\leq i\leq g-1
\]
and
\[\left\{\displaystyle
p^{c_{g,\nu,\varepsilon}}\sigma^{-1}(p)\prod_{\varepsilon(j)=-1}\chi_{j}^{-1}(p)\right\}\text{ for }i=g.
\]
Here
\[c_{i,\nu,\varepsilon}=\sum_{\nu(j)>i}(n+1-j)+2\sum_{\nu(j)\leq
  i,\varepsilon(j)=-1}(n+1-j)-g(g+1)/2\]
and
\[
c_{g,\nu,\varepsilon}=\sum_{\varepsilon(j)=-1}(n+1-j)-g(g+1)/4.
\]
\end{lemma}
\begin{proof}
The proof is inspired by \cite[Proposition
3.2.1]{genestier-tilouine}. Indeed, the Iwahori decomposition
gives $\GSp(2g, \QQ_p) = \cup_{w\in W} Bw\Iw$. Writing
$e_w$ for $w\in W$ the function defined on the cell $Bw\Iw$
gives a basis for $\mu^{\Iw}$ and the fact that the Hecke action
is upper triangular with respect to the basis $\{e_w\}$ in the
Bruhat ordering is the first part of the proof of
\cite[Proposition 3.2.1]{genestier-tilouine}. The computation of
the diagonal elements only requires replacing the parahoric
subgroup by the Iwahori subgroup and therefore the quotient of
the Weyl group by the Weyl group of the parabolic with the whole
Weyl group.

Let $\widetilde{\chi}=\chi\otimes\delta^{1/2}$ where $\delta$ is
the modulus character of the upper triangular Borel subgroup,
given by $\delta(t) =|\rho(t)|^2_p$. Write $^w
\widetilde{\chi}(t)=\delta^{1/2}(t)\delta^{1/2}(w^{-1}tw)\widetilde{\chi}(w^{-1}tw)$
and $^w\chi(t)=\chi(w^{-1}tw)$. Let $S$ be the Satake
isomorphism from the Iwahori--Hecke algebra to the spherical
Hecke algebra of the maximal torus $T$. Analogously to the
second part of \cite[Proposition 3.2.1]{genestier-tilouine}, the
eigenvalue of $U_t$ on the one-dimensional subspace of
$\mu^{\Iw}$ generated by $e_w$ is equal to $^w
\widetilde{\chi}(S(t))=\delta^{1/2}(t)\chi(w^{-1}tw)$. We now
make explicit these eigenvalues. Recall that
$w^{-1}tw=(p^{a'_{\nu(1)}}, \ldots, p^{a'_{\nu(g)}}; p^{a_0})$
where $a'_i = a_i$ if $\varepsilon(i)=1$ and $a'_i=a_0-a_i$ if
$\varepsilon(i)=-1$. Moreover, $\delta^{1/2}(t) =
p^{g(g+1)/4a_0-\sum (n+1-j)a_j}$. Thus the $e_w$ eigenvalue of
$U_t$ is
\[p^{g(g+1)/4a_0-\sum (n+1-j)a_j}\sigma(p)^{a_0}\prod_{j=1}^g \chi_j(p)^{a_{\nu(j)}\textrm{ or }a_0-a_{\nu(j)}}\]
where the exponent depends on whether $\varepsilon(\nu(j))=1$ or $-1$.

We remark that in the final computation of $\mathcal{L}$-invariants the constants $c_{i,\nu,\varepsilon}$ do not matter.
\end{proof}

\begin{lemma}\label{l:siegel-hilbert noncritical slope}
Let $\pi$ be a cuspidal automorphic representation of $\GSp(2g,
\AA_F)$ of cohomological weight $\lambda_v=\bigoplus
(\mu_{v,1},\ldots, \mu_{v,g};\mu_0)$ as above. Suppose $\nu$ is
a $p$-stabilization of $\pi$ and $t\in
\Res_{F/\QQ}T(\QQ_p)$ with associated Hecke
operator $U_{v,t}=[\Iw t\Iw]$ acting on $\pi_v$. There exists an integer $m$ such that the refinement $\nu\otimes|\cdot|_{\AA_F}^m$ of $\pi\otimes|\cdot|_{\AA_F}^m$ has noncritical slope with respect to $U_t$.
\end{lemma}
\begin{proof}
Recall Urban's convention that
$\theta(U_{v,t}) = |\lambda_v(t)|_p^{-1}a_{v,t}$ where $a_{v,t}$ is the $U_{v,t}$
eigenvalue on the refinement $\nu$. The condition that $\nu$ have
noncritical slope is that
\[\sum_{v\mid
  p}v_p(\theta(U_{v,t}))<(\lambda(\alpha^\vee)+1)v_p(\alpha(t))\]
for every simple root $\alpha$ of
$\Res_{F/\QQ}\GSp(2g)$. This is equivalent to
\[\sum_{v\mid
  p}\left(v_p(\lambda_v(t))+v_p(\alpha_{v,t})\right)<\min
\left((\mu_{v,i}-\mu_{v,i+1}+1)(a_i-a_{i+1}),
2(2\mu_{v,g}+1)a_g\right)\]
Suppose one replaces $\pi$ by $\pi_m=\pi\otimes|\cdot|^m$ for an
integer $m$. Then the cohomological weight of $\pi_m$ is
$\lambda_m'=\oplus_{v\mid p} (\mu_{v,i}; \mu_0+m)$ which implies
that $v_p(\lambda_v'(t))=v_p(\lambda_v(t))+ma_0/2$. 
$v_p(\alpha_{v,t}') = v_p(\alpha_{v,t})+m$. Lemma \ref{l:hecke local gsp} implies that $v_p(\alpha_{v,t}')=v_p(\alpha_{v,t})-ma_0$ because $(\chi_1\times\cdots\times\chi_g\rtimes\sigma)\otimes\eta\cong \chi_1\times\cdots\times\chi_g\rtimes\sigma\eta$. The result is now immediate as the right-hand side of the inequality does not change with $m$.
\end{proof}

Before discussing analytic Galois representations over Siegel
eigenvarieties we explain how to attach Galois representations to cohomological cuspidal automorphic representations of
$\GSp(2n, \AA_F)$ for a totally real $F$ using the
endoscopic classification of cuspidal representations of
symplectic groups due to Arthur, which is conditional on the
stabilization of the twisted trace formula.
\begin{theorem}\label{t:galois gsp}
Let $\pi$ be a cuspidal representation of
$\GSp(2n, \AA_F)$ of cohomological weight $\oplus (\mu_{v,1},\ldots, \mu_{v,n};\mu_0)$.
\begin{enumerate}
\item If $n=2$ there exists a spin Galois representation
$\rho_{\pi,\spin,p}:G_F\to \GSp(4, \overline{\QQ}_p)$ such that if $v\nmid\infty p$ with $\pi_v$ is unramified then $\WD(\rho_{\pi,\spin,p}|_{G_{F_v}})^{\Frss}\cong \rec_{\GSp(4)}(\pi_v\otimes|\cdot|^{-3/2})$. If $v\mid p$ and $\pi_v$ is unramified then the crystalline representation $\rho_{\pi,\spin,p}|_{G_{F_v}}$ has Hodge--Tate weights $(\mu_0-\mu_{v,1}-\mu_{v,2})/2+\{0,\mu_{v,2}+1, \mu_{v,1}+2, \mu_{v,1}+\mu_{v,2}+3\}$.
\item If $n\geq 2$ there exists a standard Galois representation $\rho_{\pi,\std, p}:G_F\to \GL(2n+1,\CC)$ such that if $v\nmid \infty p$ with $\pi_v$ the unramified principal series $\chi_1\times\cdots\times\chi_g\rtimes\sigma$ then $\rho_{\pi,\std,p}|_{G_{F_v}}$ is unramified and local-global compatibility is satisfied. If $v\mid p$ and $\pi_v$ is the unramified principal series $\chi_1\times\cdots\times\chi_g\rtimes\sigma$ then $\rho_{\pi,\std,p}|_{G_{F_v}}$ is crystalline with Hodge--Tate weights $0, \pm(\mu_{v,i}+n+1-i)$ and the eigenvalues of $\phi_{\cris}$ are $\chi_1(p),\ldots,\chi_n(p),1,\chi_1^{-1}(p),\ldots,\chi_n^{-1}(p)$.
\end{enumerate}
\end{theorem}
\begin{proof}
The first part follows from \cite[Theorem
3.5]{mok:siegel-hilbert}.

Let $\pi_0$ be any irreducible constituent of $\pi|_{\Sp(2n,
  \AA_F)}$. If $v\mid\infty$ then $\pi_{0,v}$ will then
be a discrete series representation with $L$-parameter uniquely
defined by
\[z\mapsto \diag((z/\overline{z})^{\mu_{v,i}+n+1-i}, 1, (z/\overline{z})^{-\mu_{v,i}-(n+1-i)})\]
cohomological weight $\oplus (\mu_{v,i})$.

Arthur's endoscopic classification implies the existence of a
transfer of $\pi_0$ from $\Sp(2n)$ to $\GL(2n+1)$ as
follows (cf. \cite[Corollary V.1.7]{scholze:torsion}). Let $\eta:{}^L\!\Sp(2n)\to{}^L\!\GL(2n+1)$ be the standard
inclusion of $\SO(2n+1,\CC)\into
\GL(2n+1,\CC)$. There exists a partition $2n+1=\sum_{i=0}^r\ell_i k_i$ and cuspidal automorphic representations $\Pi_i$ of
$\GL(k_i,\AA_F)$ such that:
\begin{enumerate}
\item $\Pi_i^\vee\cong\Pi_i$ for all $i$;
\item writing $\phi_{\pi_{0,v}}$ for the $L$-parameter of
$\pi_{0,v}$ and $\phi_{\Pi_{i,v}}$ for the $L$-parameter of
$\Pi_{i,v}$, if $v$ is archimedean or $\pi_{0,v}$ is unramified
then
\[\eta\circ
\phi_{\pi_{0,v}}=\oplus_{i=1}^r\oplus_{j=1}^{\ell_i}\phi_{\Pi_{i,v}}|\cdot|_v^{j-(\ell_i+1)/2}\]
\end{enumerate}
Next, if $\tau$ is a cuspidal automorphic representation of
$\GL(k,\AA_F)$ with cohomological weight $\oplus
(a_{v,1},\ldots, a_{v,k})$ and $\tau^\vee\cong\tau \otimes\chi$
for some Hecke character $\chi$ such that $\chi_v(-1)$ is
independent of $v$ then by \cite[Theorem 1.1]{blght:calabi-yau-2} there exists a continuous Galois
representation $\rho_\tau:G_F\to \GL(k,\overline{\QQ}_p)$
such that:
\begin{enumerate}
\item if $\tau_v$ is unramified and $v\nmid p$ then
$\WD(\rho_\tau|_{G_{F_v}})^{\Frss}\cong
\rec(\pi_v\otimes|\cdot|^{(1-k)/2})$;
\item if $\tau_v$ is unramified and $v\mid p$ then $\rho_\tau|_{G_{F_v}}$ is crystalline with Hodge--Tate weights $-a_{v,k+1-i}+k-i$ and $\WD(\rho_\tau|_{G_{F_v}})^{\Frss}\cong
\rec(\pi_v\otimes|\cdot|^{(1-k)/2})$. (The Hodge--Tate weights are
$(k-1)/2+$ the Harish-Chandra parameter.)
\end{enumerate}
The discrepancy between the description above and \cite[Theorem 1.1]{blght:calabi-yau-2} arises because we defined cohomological weight as the highest weight of the algebraic representation with the same central and infinitesimal characters as the discrete series, whereas in \cite[Theorem 1.1]{blght:calabi-yau-2} one uses the dual.

We will apply this to (a suitable twist of) $\Pi_i$ and denote $\rho_{\Pi_i}$ the resulting Galois representation. We define
\[\rho_{\pi,\std,p}=\rho_{\pi_0}=\oplus_{i=1}^r\oplus_{j=1}^{\ell_i}\rho_{\Pi_i}\otimes|\cdot|^{(k_i-1)/2 + j - (\ell_i+1)/2}\]
First, from the description of Hodge--Tate weights above in the case of $\GL(k)$ it is immediate that the Hodge--Tate weights of the crystalline representation $\rho_{\pi_0,p}|_{G_{F_v}}$ for $v\mid p$ are the entries of the Harish-Chandra parameter of $\pi_{0,v}$ for the corresponding $v\mid\infty$. Each Galois representation $\rho_{\Pi_i}$ was twisted by $|\cdot|^{(k_i-1)/2}$ so that local-global compatibility holds without twisting. As a result, the eigenvalues of $\phi_{\cris}$ on $\D_{\cris}(\rho_{\pi_0,p}|_{G_{F_v}})$ are $\chi_1(p),\ldots,\chi_n(p),1,\chi_1^{-1}(p),\ldots,\chi_n^{-1}(p)$, as desired. The statement at $v\nmid p$ with $\pi_{v}$ unramified is analogous.

\end{proof}

\begin{lemma}\label{l:gsp eigenvariety}
Let $\pi$ be a cohomological cuspidal automorphic representation of $\GSp(2g,\AA_F)$, $\nu$ a $p$-stabilization of $\pi$ of noncritical slope and
$\mathcal{E}\to \mathcal{W}$ Urban's eigenvariety around
$(\pi,\nu)$. If $\rho_{\pi,\std,p}$ is irreducible, shrinking $\mathcal{W}$, there exists an analytic
Galois representation $\rho_{\mathcal{E},\std}:G_F\to \GL(2g+1,
\mathcal{O}_{\mathcal{E}})$ such that for $z\in \Sigma$, $z\circ
\rho_{\mathcal{E},\std}=\rho_{\pi_z,\std,p}$. If $\pi$ has
Iwahori level at $v\mid p$ then $\rho_{\mathcal{E},\std}$ admits
a refinement in the sense of \S \ref{sect:refined galois}.

In the case $g=2$, if $\rho_{\pi,\spin,p}$ is irreducible one obtains an analogous analytic spin Galois representation $\rho_{\mathcal{E},\spin}:G_F\to \GSp(4, \mathcal{O}_{\mathcal{E}})$ which admits a refinement if $\pi_v$ is Iwahori spherical at $v\mid p$.
\end{lemma}
\begin{proof}
We will follow the proof of Lemma \ref{l:hmf eigenvariety}. First, the existence of $\rho_{\mathcal{E},\spin}$ and $\rho_{\mathcal{E},\std}$ follows analogously.

Next, we need to show that if $\pi_v$ is Iwahori spherical at $v\mid p$ then there is a dense set of points $\Sigma'\subset\Sigma\subset \mathcal{E}$ such that if $z\in \Sigma'$ then $\pi_{z,v}$ is unramified at $v\mid p$. 

Let $z\in \Sigma$ and let $\pi_z$ be the associated cuspidal
representation. Suppose $\pi_{z,v}$ is not unramified. Since it
is Iwahori spherical it follows from \cite[Theorem
7.9]{tadic:symplectic} that for every $v\mid p$,
$\pi_{z,v}=\Ind\chi$ where
$\chi=\chi_1\times\cdots\times \chi_g\rtimes\sigma$ such that one of
the following is satisfied:
\begin{enumerate}
\item $\chi_i^2=1$ but $\chi_i\neq 1$ for at least 3 indices
$i$,
\item $\chi_i=|\ \ |^{\pm 1}$ for at least one index $i$, or
\item $\chi_i\chi_j^{\pm 1}=|\ \ |$ for at least one pair $(i,j)$ and choice of exponent.
\end{enumerate}
Denote by $\alpha_{v,i}$ the eigenvalue of $[\Iw\beta_{g-i}\Iw]$ acting
on $\pi_{z,v}$. There exists a permutation $\nu$ and a function
$\varepsilon:\{1,\ldots,g\}\to \{-1,1\}$ such that $\alpha_{v,i}$ are
given by Lemma \ref{l:hecke local gsp}. Solving, one obtains
\[\chi_{\nu^{-1}(i)}(p)^{\varepsilon(i)}=p^{c_{i,\nu,\varepsilon}-c_{i-1,\nu,\varepsilon}}\frac{\alpha_{v,i-1}}{\alpha_{v,i}}\]
for $1<i<g$, $\chi_{\nu^{-1}(g)}(p)=p^{2c_{g,\nu,\varepsilon}-c_{g-1,\nu,\varepsilon}}\alpha_{v,g-1}/\alpha_{v,g}^2$
and
$\chi_{\nu^{-1}(1)}(p)^{\varepsilon(1)}=p^{\mu_0-c_{1,\nu,\varepsilon}}\alpha_{v,1}$
where the last equality comes from the fact that
$\det\pi_{z,v}(p)=\chi_1\cdots\chi_g\sigma^2(p)=p^{\mu_0}$.

But $\alpha_{v,i}=|\lambda_v(\beta_{g-i})|_p\cdot\theta(U_{v,i})$
where
$\lambda_v(\beta_{g-i})=p^{\mu_{v,1}+\cdots+\mu_{v,i}-\mu_0}$ for $1\leq i<g$ and $\lambda_v(\beta_0)=p^{(\mu_{v,1}+\cdots+\mu_{v,g}-\mu_0)/2}$. We
deduce that
\[\chi_{\nu^{-1}(i)}(p)^{\varepsilon(i)}=p^{c_{i,\nu,\varepsilon}-c_{i-1,\nu,\varepsilon}+\mu_{v,i}}\frac{\theta(U_{v,i-1})}{\theta(U_{v,i})}\]
for $1<i<g$,
$\chi_{\nu^{-1}(g)}(p)=p^{2c_{g,\nu,\varepsilon}-c_{g-1,\nu,\varepsilon}+\mu_{v,g}} \theta(U_{v,g-1})/\theta(U_{v,g})^2$
and $\chi_{\nu^{-1}(1)}(p)^{\varepsilon(1)}=p^{\mu_0-c_{1,\nu,\varepsilon}+\mu_{v,1}-\mu_0}\theta(U_{v,1})= p^{-c_{1,\nu,\varepsilon}+\mu_{v,1}}\theta(U_{v,1})$.

The functions $\theta(U_{v,i})$ are analytic on $\mathcal{E}$
and so $v_p(\theta(U_{v,i}))$ is locally constant. By shrinking
$\mathcal{E}$ we may even assume they are constant. Since
$\pi_{z,v}$ is Iwahori spherical but ramified it follows from
the conditions listed above that $v_p(\chi_i(p))\in \{0,1\}$ or
$v_p(\chi_i(p)\chi_j(p)^{\pm1})=1$. This, however, implies certain linear combinations of the weights $\mu_{v,i}$ and $\mu_0$ are constant, which is a contradiction as $\mathcal{E}$ maps to a full-dimensional open set in the weight space.

Suppose $z\in \Sigma'$ in which case
$\rho_{\pi_z,\std,p}|_{G_{F_v}}$ is crystalline for $v\mid
p$. Consider the analytic functions
$a_{v,i}(z)=z(\theta(U_{v,i}))$ on $\mathcal{E}$. Let
$\kappa_{n+1}(z) = 0$, $\kappa_{n+1\pm
  i}(z)=\pm(\mu_{v,n+1-i}(z)+i)$ for $1\leq i\leq n$; Theorem
\ref{t:galois gsp} these are the Hodge--Tate weights, arranged
increasingly, of $z\circ\rho_{\mathcal{E}}|_{G_{F_v}}$. By
local-global compatibility the eigenvalues of $\phi_{\cris}$
acting on $\D_{\cris}(\rho_{\pi_z,p}|_{G_{F_v}})$ are
$\chi_i(p)^{\pm 1}$ and $1$. We will use the formulae above to
construct the analytic functions $F_k$.

Let $F_{n+1}(z)=1$. For $1<i<n$ let
\[F_{n+1\pm (n+1-i)}(z)=\left(p^{c_{i,\nu,\varepsilon}-c_{i-1,\nu,\varepsilon}-i}\frac{a_{v,i-1}(z)}{a_{v,i}(z)}\right)^{\pm
1}.\]
Let $F_{n+1\pm (n+1-n)}(z) = \left(p^{2c_{g,\nu,\varepsilon}-c_{g-1,\nu,\varepsilon}-n} a_{v,g-1}/a_{v,g}^2\right)^{\pm 1}$ and $F_{n+1\pm (n+1-1)}(z)=\left(p^{-c_{1,\nu,\varepsilon}-1}a_{v,1}\right)^{\pm 1}$. Thus $p^{\kappa_{n+1\pm i}}F_{n+1\pm i}=\chi_{\nu^{-1}(i)}(p)^{\pm \varepsilon(i)}$ for $1<i\leq n$ and $p^{\kappa_{n+1\pm 1}}F_{n+1\pm 1}=\chi_{\nu^{-1}(1)}(p)^{\pm 1}$. By Theorem \ref{t:galois gsp} these are the eigenvalues of $\phi_{\cris}$ and so $\rho_{\mathcal{E},\std,p}$ admits a refinement.

Finally, we need to construct a refinement for $\rho_{\mathcal{E},\spin,p}$ in the genus $n=2$ case. The eigenvalues of $\phi_{\cris}$ in this case acting on $\chi_1\times\chi_2\rtimes\sigma$ are $p^{-3/2}\times\{\sigma(p), \sigma(p)\chi_1(p), \sigma(p)\chi_2(p), \sigma(p)\chi_1(p)\chi_2(p)\}$. Let $\kappa_1 = (\mu_0-\mu_{v,1}-\mu_{v,2})/2$, $\kappa_2 = (\mu_0-\mu_{v,1}+\mu_{v,2})/2+1$, $\kappa_3 = (\mu_0+\mu_{v,1}-\mu_{v,2})/2+2$ and $\kappa_4=(\mu_0+\mu_{v,1}+\mu_{v,2})/2+3$. For simplicity of notation we will assume that $\nu=1$ and $\varepsilon=1$, the other cases being analogous. (Later we will choose this refinement anyway.) Then $\sigma(p)=p^{c_2+(\mu_0-\mu_{v,1}-\mu_{v,2})/2}a_{v,2}^{-1}$, $\chi_2(p)=p^{\mu_{v,1}-c_1}a_{v,1}$ and $\chi_1(p) = p^{c_1-2c_2+\mu_{v,2}}a_{v,2}^{2}a_{v,1}^{-1}$. Write $F_1=p^{c_2-3/2}a_{v,2}^{-1}$, $F_2=p^{c_1-c_2-1-3/2}(a_{v,2}/a_{v,1})$, $F_3=p^{c_2-c_1-2-3/2}(a_{v,1}/a_{v,2})$ and $F_4=p^{-c_2-3-3/2}a_{v,2}$ which are analytic and satisfy $p^{\kappa_1}F_1=p^{-3/2}\sigma(p)$, $p^{\kappa_2}F_2=p^{-3/2}\sigma(p)\chi_1(p)$, $p^{\kappa_3}F_3=p^{-3/2}\sigma(p)\chi_2(p)$ and $p^{\kappa_4}F_4=p^{-3/2}\sigma(p)\chi_1(p)\chi_2(p)$, which are the eigenvalues of $\phi_{\cris}$. Thus $\rho_{\mathcal{E},\spin,p}$ has a refinement.
\end{proof}

\begin{corollary}\label{c:gsp4 triangulation}
Let $\pi$ be a Hilbert modular form of infinity type $(k_1,\ldots, k_d;w)$. Suppose $\pi$ is not CM and let $\Pi$ be the cuspidal representation of $\GSp(4,\AA_F)$ from Theorem \ref{t:GL(4) to GSp(4)}. Let $\nu$ be a $p$-stabilization of $\Pi$ and let $m\in \ZZ$ such that $\Pi\otimes|\cdot|^m$ has noncritical slope. For $v\mid
p$, assume that $\pi_v$ is Iwahori spherical; if $\pi_v$ is an
unramified principal series assume that $\alpha_v/\beta_v\notin\mu_{60}$. Then $\rho_{\mathcal{E},\spin,p}$ has a global triangulation whose graded pieces $\mathcal{R}(\delta_i)$ are such that $\delta_i(u)=u^{\kappa_i}$ and $\delta_i(p)=F_i$ from the proof of Lemma \ref{l:gsp eigenvariety}.
\end{corollary}
\begin{proof}
Since the Hodge--Tate weights in the triangulation are ordered increasingly the only thing left to check is that the associated refinement is regular, i.e., that $\det\varphi$ on the filtered piece $\mathcal{F}_i$ has multiplicity one in $\D_{\cris}(\wedge^i \rho_{\Pi})$. This is equivalent to showing that for each $i\in \{1,2,3,4\}$, each product of $i$ terms in $\{\alpha_v^3,\alpha_v^2 \beta_v, \alpha_v \beta_v^2, \beta_v^3\}$ occurs once, which can be checked if $\alpha_v/\beta_v\notin\mu_{60}$. 
\end{proof}

\begin{corollary}\label{c:gsp triangulation}
Let $\pi$ be a Hilbert modular form of infinity type $(k_1,\ldots, k_d;w)$. Suppose $\pi$ is not CM and let $\Pi$ be the cuspidal representation of $\GSp(2n,\AA_F)$ from Theorem \ref{t:GL(2n+1) to Sp(2n)}. Let $\nu$ be a $p$-stabilization of $\Pi$ and let $m\in \ZZ$ such that $\Pi\otimes|\cdot|^m$ has noncritical slope. For $v\mid
p$, assume that $\pi_v$ is Iwahori spherical; if $\pi_v$ is an
unramified principal series assume that $\alpha_v/\beta_v\notin\mu_{\infty}$. Then $\rho_{\mathcal{E},\std,p}$ has a global triangulation whose graded pieces $\mathcal{R}(\delta_i)$ are such that $\delta_i(u)=u^{\kappa_i}$ and $\delta_i(p)=F_i$ from the proof of Lemma \ref{l:gsp eigenvariety}.
\end{corollary}
\begin{proof}
As in the previous corollary we only need to check that for each $1\leq i\leq 2n+1$ each product of $i$ terms in $\{(\alpha_v/\beta_v)^k|-n\leq k\leq n\}$ occurs only once. Again, this can be checked when $\alpha_v/\beta_v\notin\mu_\infty$. 
\end{proof}

\subsection{Eigenvarieties for  unitary groups}\label{sect:eigenvarieties for unitary}
One could reproduce the results of \S \ref{sect:eigenvarieties
  for symplectic} in the context of unitary groups. Indeed, the
endoscopic classification for unitary groups was completed by
Mok and compact unitary groups of course have discrete series so
all the results translate into this context, again under the
assumption of stabilization of the twisted trace formula. The
main reason for redoing the computations using unitary groups is work in progress of Eischen--Harris--Li--Skinner and Eischen--Wan which will produce $p$-adic $L$-functions for unitary groups.

Let $F/\QQ$ be a totally real field in which $p$ splits completely and $E/F$ a CM
extension in which $p$ splits completely as well. Suppose $U$ is a definite unitary group over $F$, in
$n$ variables, attached to $E/F$. Suppose $\pi$ is an
irreducible (necessarily cuspidal) automorphic representation
$\pi$ of $U(\AA_F)$ of cohomological weight
$\bigoplus_{v\mid\infty} (\mu_{v,1},\ldots, \mu_{v,n})$. Then the
restriction to $W_{\CC}$ of the $L$-parameter of $\pi_v$ is given by
\[z\mapsto \diag((z/\overline{z})^{\mu_{v,i}+(n+1)/2-i})\]
(cf. \cite[\S 6.7]{bellaiche-chenevier:selmer}).

If for some $v\mid p$ the representation $\pi_v$ has Iwahori level then the Hecke operators $U_{v,i}=[\Iw e_i^\vee(p)\Iw]$ act on $\pi_v$ where $e_i^\vee$ is dual to the character $e_i$ isolating the $i$-th entry on $T$. For consistency of notation with the previous section we remark that Urban's $\theta(U_{v,i})$ is denoted by $\delta^{1/2}\psi_{\pi,\mathcal{R}}$ in \cite[\S 7.2.2]{bellaiche-chenevier:selmer}. If $\pi_v=\chi_1\times\cdots\times\chi_n$ is an unramified principal series then $\pi_v^{\Iw}$ is $n$ dimensional and the Hecke operators $U_{v,i}$ can be simultaneously written in upper triangular form with diagonal entries $\chi_i(p)p^{-(n-1)/2}$.

\begin{theorem}\label{t:galois unitary}
Suppose $U$ and $\pi$ are as above, with $\pi$ of cohomological
weight $\bigoplus_{v\mid\infty} (\mu_{v,1},\ldots,
\mu_{v,n})$. Then there exists a continuous Galois
representation $\rho_{\pi,p}:G_E\to \GL(n,
\overline{\QQ}_p)$ such that:
\begin{enumerate}
\item If $v\nmid p\infty$ and $\pi_v$ is unramified then
$\rho_{\pi,p}|_{G_{E_w}}$ is unramified for $w\mid p$ and
$\WD(\rho_{\pi,p}|_{G_{E_w}})^{\Frss}\cong\rec(\BC_{E_w/F_v}(\pi_v)\otimes|\cdot|^{-(n-1)/2})$.
\item If $v\mid p$, since it splits in $E$ we may write $v=w \overline{w}$ where $w$ is the finite place of $E$ corresponding to $\iota_p:\overline{\QQ}\to \overline{\QQ}_p$. If $\pi_v=\chi_1\times\cdots\times\chi_n$ is unramified then $\rho_{\pi,p}|_{G_{E_w}}$ is crystalline with Hodge--Tate weights $-\mu_{v,i}+i$ and $\phi_{\cris}$ has eigenvalues $\chi_i(p)p^{-(n-1)/2}$.
\end{enumerate}
\end{theorem}
\begin{proof}
The proof is analogous to that of Theorem \ref{t:galois gsp} as the transfer from $U$ to $\GL$ is the content of \cite{mok:functoriality1} (cf. \cite[Corollary V.1.7]{scholze:torsion}). The statement about Hodge-Tate weights follows by appealing to \cite[Theorem 1.2]{blght:calabi-yau-2} rather than \cite[Theorem 1.1]{blght:calabi-yau-2}.
\end{proof}

\begin{remark}
The literature contains base change results for both isometry
unitary and similitude unitary groups to various degrees of
generality. We remark that one may deduce base change for
isometry unitary groups from the analogous results for
similitude results using algebraic liftings of automorphic
representations (\cite[Proposition 12.3.3]{patrikis:tate}).
\end{remark}

The main theorem of \cite{chenevier:unitary-eigenvarieties} implies that if $4\mid n$, which we will asume, then the conclusion of Theorem \ref{t:urban} holds for $\pi$ and a $p$-stabilization $\nu$. Moreover, if $\rho_{\pi,p}$ is irreducible then Theorem \ref{t:galois unitary} implies the existence of an analytic Galois representation $\rho_{\mathcal{E},p}:G_E\to \GL(n, \mathcal{O}_{\mathcal{E}})$ interpolating, as before, the Galois representations attached to the classical regular points on $\mathcal{E}$. 

\begin{corollary}\label{c:unitary triangulation}
Let $F$ be a totally real field in which $p$ splits completely. Let $\pi$ be a Hilbert modular form over $F$, of infinity type $(k_1,\ldots, k_d;w)$, suppose there exist finite places $w_1, w_2$ not above $p$ with the property that $\pi_w$ is special for $w\in \{w_1,w_2\}$, and suppose that $\pi_v$ is Iwahori spherical for $v\mid p$ and that if $\pi_v$ is unramified with Satake parameters $\alpha_v$ and $\beta_v$ then $\alpha_v/\beta_v\notin\mu_\infty$. Suppose $\pi$ is not CM. Let $E$ be a CM extension of $F$, $\psi$ a Hecke character of $E$, and $\Pi$ a cuspidal automorphic representation of $U(\AA_F)$ such that $\Pi=\psi\otimes\BC_{E/F}\Sym^n\pi$ as in Proposition \ref{p:sym hmf to unitary}. Let $\mathcal{E}$ and $\rho_{\mathcal{E},p}$ as above. Then $\rho_{\mathcal{E},p}|_{G_{E_w}}$ for $w\mid v\mid p$ a finite place of $E$ admits a triangulation with graded pieces $\mathcal{R}(\delta_i)$ such that $\delta_i(u)=u^{\kappa_i}$ for $u\in \ZZ_p^\times$ and $\delta_i(p)=F_i$ where $\kappa_i=-\mu_{v,i}+i$ and $F_i=p^{(n-1)/2-i}a_{v,i}$ where $a_{v,i}=\theta(U_{v,i})$ is analytic over $\mathcal{E}$.
\end{corollary}
\begin{proof}
Theorem \ref{t:galois unitary} implies that at regular classical points which are unramified at $v\mid p$ the analytic functions $\kappa_i$ give the Hodge--Tate weights of $\rho_{\mathcal{E},p}|_{G_{E_w}}$. Thus it suffices to check that $p^{\kappa_i}F_i$ gives the eigenvalues of $\phi_{\cris}$ at such regular unramified crystalline points. This follows from the fact that the eigenvalue of $U_{v,i}$ on $\pi_v$ equals $a_{v,i}$ times $|\lambda(e_i^\vee(p))|_p^{-1}\delta^{-1/2}(e_i^\vee(p))$. Finally, the condition $\alpha_v/\beta_v\notin\mu_\infty$ implies the existence of the global triangulation as in the proof of Corollary \ref{c:gsp triangulation}.
\end{proof}

\section{Computing the $\mathcal{L}$-invariants}\label{sect:l}
Let $F$ be a totally real field in which the prime $p$ splits
completely and let $\pi$ be a non-CM cohomological Hilbert
modular form of infinity type $(k_1,\ldots,k_d;w)$. Let
$V_{2n}=\rho_{\pi,p}\otimes\det^{-n}\rho_{\pi,p}$. Suppose that
for $v\mid p$, $\pi_v$ is Iwahori spherical, which is equivalent
to the requirement that $\rho_{\pi,p}|_{G_{F_v}}$ be
semistable. For each such $v$ let $D_v\subset
\D_{\st}(V_{2n,v})$ be the regular submodule chosen in \S
\ref{sect:regular submodules}.

Under the hypotheses (C1--4), we will compute
$\mathcal{L}(V_{2n}, D)$ in terms of logarithmic derivatives of
analytic Hecke eigenvalues over eigenvarieties. We will assume
the existence of a rigid analytic space $\mathcal{E}\to
\mathcal{W}$ which is \'etale at a weight $w_0$ over which one
has a point $z_0\in \mathcal{E}$ such that
$z_0\circ\rho_{\mathcal{E},p}\cong
\psi\otimes\Sym^m\rho_{\pi,p}$ for some Hecke character $\psi$
and some $m\geq n$. We will moreover assume that $z_0$ corresponds to the $p$-stabilization of $\Pi_v$ coming from the $p$-stabilization of $\pi_v$ that gave rise to the regular submodule $D_v$. Assume there exists a global analytic
triangulation of
$\mathscr{D}_{\rig}^\dagger(\rho_{\mathcal{E},p}|_{G_{F_v}})$
with graded pieces $\mathcal{R}(\delta_i)$. 

\begin{lemma}\label{l:derivative galois}
Let $\overrightarrow{u}$ be a direction in $\mathcal{W}$ and let $\nabla_{\overrightarrow{u}}\rho_{\mathcal{E},p}$ be the tangent space to $\rho_{\mathcal{E},p}$ in the $\overrightarrow{u}$-direction, which makes sense under the assumption that $\mathcal{E}\to \mathcal{W}$ is \'etale at $z_0$. Then $c_{\overrightarrow{u}}=(z_0\circ\rho_{\mathcal{E},p})^{-1}\nabla_{\overrightarrow{u}}\rho_{\mathcal{E},p}$ is a cohomology class in $H^1(F, \End(z_0\circ\rho_{\mathcal{E},p}))$ and the natural projection $c_{\overrightarrow{u},n}\in H^1(F, V_{2n})$ lies in fact in the Selmer group $H^1(\{D_v\}, V_{2n})$.
\end{lemma}
\begin{proof}
Note that
$\End(z_0\circ\rho_{\mathcal{E},p})\cong\End(\psi\otimes\Sym^m\rho_{\pi,p})\cong\oplus_{i=0}^m
V_{2i}$ and the natural projection on cohomology arises from the
natural projection of this representation to $V_{2n}$.

One needs to check two things. The first, that the cohomology
classes are unramified at $v\notin S\cup\{w\mid p\}$ follows along the same lines as \cite[Lemma 1.3]{hida:mazur-tate-teitelbaum}. The second is that the image of the cohomology class in $H^1(F_v,V_{2n})/H^1_f(F_v, V_{2n})$ lands in $H^1(F_1\Drig(V_{2n,v}))/H^1_f(F_v, V_{2n})$. But Proposition \ref{p:End to Sym^n} implies that the natural projection $c_{\overrightarrow{u},n}$ (in the notation of this lemma) lies entirely in the span of $e_1^ie_2^{n-i}$ for $2i\leq n$. By the choice of regular submodular $D_v$, this implies that $\res_vc_{\overrightarrow{u},n}\in H^1(F_1\Drig(V_{2n,v}))$, as desired.
\end{proof}

\begin{proposition}\label{p:l-invariant formula}
In the notation of the previous lemma, suppose
$\mathscr{D}_{\rig}^\dagger(\rho_{\mathcal{E},p}|_{G_{F_v}})$
has an analytic triangulation with graded pieces
$\delta_1,\ldots, \delta_{m+1}$ where $\delta_i(u)=u^{\kappa_i}$
and $\delta_i(p)=p^{\kappa_i}F_i$. Then
\[\mathcal{L}(V_{2n}, \{D_v\})=\prod_{v\mid p}\left(-\frac{\sum_i B_{m,n,i-1}(\nabla_{\overrightarrow{u}}F_i)/F_i}{\sum_i B_{m,n,i-1}\kappa_i}\right)\]
as long as this formula makes sense. Here, the coefficients $B_{m,n,i}$ are given in Remark \ref{r:Bnki} and are basically alternating binomial coefficients multiplied by inverse Clebsch--Gordon coefficients.
\end{proposition}
\begin{proof}
Lemma \ref{l:rank 1 L invariant} implies that
$\mathcal{L}(V_{2n}, \{D_v\})=\prod_{v\mid p} (a_v/b_v)$ where
the projection $c_{\overrightarrow{u}, v}$ to
$H^1(\gr_1\Drig(V_{2n,v}))\cong H^1(\mathcal{R})$ is written as
$a_v(-1,0)+b_v(0,\log_p\chi(\gamma))$. Explicitly,
\[\frac{a_v}{b_v}=\frac{c_{\overrightarrow{u},
    v}(p)}{c_{\overrightarrow{u}, v}(u)/-\log_p(u)}\]
for $u\in \ZZ_p^\times$. But Proposition \ref{p:End to Sym^n}
implies that
\[c_{\overrightarrow{u}, v}=\sum
B_{m,n,i}(\nabla_{\overrightarrow{u}}\delta_i)/\delta_i\]
The result now follows from the fact that $\delta_i(p)=F_i$, $\delta_i(u)=u^{\kappa_i}$ and $(\nabla_{\overrightarrow{u}}u^{\kappa})/u^{\kappa} = \nabla_{\overrightarrow{u}}\kappa \log_p(u)$.
\end{proof}

In the remaining sections we apply Proposition
\ref{p:l-invariant formula} to obtain explicit formulae for
$\mathcal{L}$-invariants for relevant symmetric powers in terms
of logarithmic derivatives of analytic Hecke eigenvalues.

We will assume that $F$ is a totally real field in which $p$ splits completely and $\pi$ is a cohomological Hilbert modular form with infinity type $(k_1,\ldots, k_d;w)$. At $v\mid p$ we assume that $\pi_v$ is Iwahori spherical. In the computation of the $\mathcal{L}$-invariant of $V_{2n}$ we will assume that $H^1_f(F, V_{2n})=0$. Throughout we will consider the refinement of $\pi$ corresponding to the ordering $e_1,e_2$ of the basis of $\D_{\st}(\rho_{\pi,p}|_{G_{F_v}})$, which gives a suitable refinement of any automorphic form equivalent to $\Sym^m \pi$ using the ordering $e_1^m, e_1^{m-1}e_2, \ldots, e_2^m$. We will assume that $V_{2n,v}$ satisfies condition (C4). 

\subsection{Symmetric squares}\label{sect:sym2l}
Suppose for $v\mid p$ such that $\pi_v$ is unramified that the
two Satake parameters are distinct. Let $\mathcal{E}$ be the eigenvariety from Lemma \ref{l:hmf eigenvariety}. Suppose that $\mathcal{E}$ is \'{e}tale over the weight space at the chosen refinement of $\pi_v$.

\begin{theorem}\label{t:l invariant formula hmf}Writing $a'_v$ for the derivative in the direction $(1,\ldots, 1; -1)$ in the weight space we compute
\[\mathcal{L}(V_2, \{D_v\}) = \prod_{v\mid p}
\left(\frac{-2a_v'}{a_v}\right)\]
\end{theorem}
\begin{proof}
Recall from Corollary \ref{c:hmf triangulation} that $\kappa_1=(w-k_v)/2$, $\kappa_2=(w+k_v-2)/2$, $F_1=a_v^{-1}p^{3/2}$ and $F_2=a_vp^{1/2}$. The result now follows directly from Proposition \ref{p:l-invariant formula} and the fact that $B_{1,1,0}=1$ and $B_{1,1,1}=-1$.
\end{proof}

\subsection{Symmetric sixth powers}\label{sect:sym6l}
Assume that $\pi$ is not CM. Suppose for $v\mid p$ such that
$\pi_v$ is unramified that $\alpha_v/\beta_v\notin\mu_{60}$. Let
$\Pi$ be a suitably twisted Ramakrishnan--Shahidi lift of $\Sym^3\pi$ such that the chosen refinement has noncritical slope (cf. Theorem \ref{t:GL(4) to GSp(4)} and Lemma \ref{l:siegel-hilbert noncritical slope}). Let $\mathcal{E}$ be Urban's eigenvariety for $\GSp(4)$ and let $a_{v,1}$ and $a_{v,2}$ be the analytic Hecke eigenvalues from the proof of Lemma \ref{l:gsp eigenvariety}. Suppose that the eigenvariety $\mathcal{E}$ is \'{e}tale over the weight space at the chosen refinement of $\Pi$. 

\begin{theorem}\label{t:l invariant formula gsp4}
If $\overrightarrow{u}=(u_1,u_2; u_0)$ is any direction in
the weight space, i.e. $u_1\geq u_2\geq 0$, such that the denominator below is non-zero,
then
\[\mathcal{L}(V_6, \{D_v\})=\prod_{v\mid p}
\left(\frac{-4\widetilde{\nabla}_{\overrightarrow{u}}a_{v,2}+3\widetilde{\nabla}_{\overrightarrow{u}}a_{v,1}}{u_1-2u_2}\right)\]
where we write $\widetilde{\nabla}_{\overrightarrow{u}}f =
(\nabla_{\overrightarrow{u}}f)/f$. 
\end{theorem}
\begin{proof}
Recall from Corollary \ref{c:gsp4 triangulation} that
$\kappa_1=(\mu_0-\mu_{v,1}-\mu_{v,2})/2$,
$\kappa_2=(\mu_0-\mu_{v,1}+\mu_{v,2})/2+1$,
$\kappa_3=(\mu_0+\mu_{v,1}-\mu_{v,2})/2+2$,
$\kappa_4=(\mu_0+\mu_{v,1}+\mu_{v,2})/2+3$ giving
$\nabla_{\overrightarrow{u}}\kappa_1=(u_0-u_1-u_2)/2$,
$\nabla_{\overrightarrow{u}}\kappa_2=(u_0-u_1+u_2)/2$,
$\nabla_{\overrightarrow{u}}\kappa_3=(u_0+u_1-u_2)/2$,
$\nabla_{\overrightarrow{u}}\kappa_4=(u_0+u_1+u_2)/2$. Similarly,
$\widetilde{\nabla}_{\overrightarrow{u}}F_1=-\widetilde{\nabla}_{\overrightarrow{u}}a_{v,2}$,
$\widetilde{\nabla}_{\overrightarrow{u}}F_2=\widetilde{\nabla}_{\overrightarrow{u}}a_{v,2}-\widetilde{\nabla}_{\overrightarrow{u}}a_{v,1}$,
$\widetilde{\nabla}_{\overrightarrow{u}}F_3=\widetilde{\nabla}_{\overrightarrow{u}}a_{v,1}-\widetilde{\nabla}_{\overrightarrow{u}}a_{v,2}$,
$\widetilde{\nabla}_{\overrightarrow{u}}F_4=\widetilde{\nabla}_{\overrightarrow{u}}a_{v,2}$. Using
that $(B_{3,3,i})_i\sim (1,-3,3,-1)$ we deduce the formula.
\end{proof}

\subsection{Symmetric powers via symplectic groups}\label{sect:symgspl}
We remark that the results of this paragraph are conditional on the stabilization of the twisted trace formula (cf. Theorem \ref{t:galois gsp}). 

Assume that $\pi$ is not CM. Suppose for $v\mid p$ such that
$\pi_v$ is unramified that
$\alpha_v/\beta_v\notin\mu_{\infty}$. Suppose $\pi$ satisfies the hypotheses of Theorem \ref{t:automorphy of sym} (2) and let
$\Pi$ be a suitable (as before, from Lemma \ref{l:siegel-hilbert noncritical slope}) twist of the cuspidal representation of $\GSp(2n, \AA_F)$ from Theorem \ref{t:GL(2n+1) to Sp(2n)}. Let $\mathcal{E}$ be Urban's eigenvariety for $\GSp(2n)$ and let $a_{v,i}$ be the analytic Hecke eigenvalues from the proof of Lemma \ref{l:gsp eigenvariety}. Suppose that the eigenvariety $\mathcal{E}$ is \'{e}tale over the weight space at the chosen refinement of $\Pi$. 

\begin{theorem}\label{t:l invariant formula gsp}
If $\overrightarrow{u}=(u_1,\ldots,u_n;u_0)$ is any direction in
the weight space, such that the denominator below is non-zero, then
\[\mathcal{L}(V_{4n-2}, \{D_v\})=\prod_{v\mid
  p}-\left(\frac{B_{n}\widetilde{\nabla}_{\overrightarrow{u}}a_{v,1}+B_{1}(\widetilde{\nabla}_{\overrightarrow{u}}a_{v,n-1}-2\widetilde{\nabla}_{\overrightarrow{u}}a_{v,n})+\sum_{i=2}^{n-1}B_{i}(\widetilde{\nabla}_{\overrightarrow{u}}a_{v,i-1}-\widetilde{\nabla}_{\overrightarrow{u}}a_{v,i})}{\sum_{i=1}^nu_iB_{n+1-i}}\right)\]
where we write $B_i=(-1)^i\binom{2n}{n+i}i$.
\end{theorem}
\begin{remark}
Given that Remark \ref{r:Bnki} gives explicit values for the $B_{n,k,i}$, we know there will be directions where the denominator is non-zero.
\end{remark}

\begin{proof}
Combine the formulae for the triangulation of the standard
Galois representation from the proof of Lemma \ref{l:gsp
  eigenvariety} with Proposition \ref{p:l-invariant
  formula}. Finally, compute
\begin{align*}
B_i&=B_{2n,2n-1,n+i}-B_{2n,2n-1,n-i}\\
&=(-1)^{n+i}\binom{2n}{n+i}(2n-2(n+i))-(-1)^{n-i}\binom{2n}{n-i}(2n-2(n-i))\\
&=(-1)^{n+1}2 \cdot(-1)^i\binom{2n}{n+i}i
\end{align*}
and the result follows because all the $B_i$ can be scaled by the same factor.
\end{proof}

\subsection{Symmetric powers via unitary groups}\label{sect:symunitaryl}
We remark that the results of this paragraph as well are conditional on the stabilization of the twisted trace formula.

Assume that $\pi$ is not CM. Suppose for $v\mid p$ such that
$\pi_v$ is unramified that
$\alpha_v/\beta_v\notin\mu_{\infty}$. Suppose $\pi$ satisfies
the hypotheses of Theorem \ref{t:automorphy of sym} (2) and Proposition \ref{p:sym hmf to unitary}. Let $E/F$ the CM extension and $\Pi$ the cuspidal representation of $U_{4n}(\AA_F)$ which a transfer of a twist of $\Sym^{4n-1}\pi$ as in Proposition \ref{p:sym hmf to unitary}. Let $\mathcal{E}$ be Chenevier's eigenvariety and let $a_{v,i}$ be the analytic Hecke eigenvalues from the proof of Corollary \ref{c:unitary triangulation}. Suppose that the eigenvariety $\mathcal{E}$ is \'{e}tale over the weight space at the chosen refinement of $\Pi$. 

\begin{theorem}\label{t:l invariant formula unitary}
If $\overrightarrow{u}=(u_1,\ldots,u_n;u_0)$ is any direction in
the weight space, then
\[\mathcal{L}(V_{8n-2}, \{D_v\})=\prod_{v\mid
  p}\left(\frac{-\sum_{i=1}^{4n}(-1)^{i-1}\binom{4n-1}{i-1}\widetilde{\nabla}_{\overrightarrow{u}}a_{v,i}}{\sum_{i=1}^{4n}(-1)^{i-1}\binom{4n-1}{i-1}u_i}\right)\]
and
\[\mathcal{L}(V_{8n-6}, \{D_v\})=\prod_{v\mid
  p}\left(\frac{-\sum_{i=1}^{4n}B_{i-1}\widetilde{\nabla}_{\overrightarrow{u}}a_{v,i}}{\sum_{i=1}^{4n}u_iB_{i-1}}\right)\]
Here $B_i=B_{4n-1,4n-3,i}$ is the inverse Clebsch--Gordan
coefficient of Proposition \ref{p:End to Sym^n}, up to a scalar independent of $i$ given by
\[B_i= (-1)^i\binom{4n-1}{i}((4n-1)^3-(4i+1)(4n-1)^2+(4i^2+2i)(4n-1)-2i^2).\]
\end{theorem}
\begin{proof}
Note that we cannot simply appeal to Proposition
\ref{p:l-invariant formula} as the analytic Galois
representation on the unitary eigenvariety is a representation
of $G_E$ and not $G_F$. However, note that since $E/F$ is finite and $V_{2m}$ is a characteristic 0 vector
space, the inflation-restriction sequence gives
$H^1(G_{F,S},V_{2m})\cong H^1(G_{E,S_E},V_{2m})^{G_{E/F}}$. We
have constructed a cohomology class $c_{\overrightarrow{u}}\in
H^1(G_{E,S_E},V_{m})$ and by construction $c_{\overrightarrow{u}}$ is
invariant under the nontrivial element of $\Gal(E/F)$ (complex conjugation acts trivially on $V_{2m}$). Thus it
descends to a cohomology class in $H^1(G_{F,S}, V_{2m})$. Since $p$ splits completely in both $E$ and $F$, this local component of the descended class is the same as that of the original class. This implies that the conclusion of Proposition
\ref{p:l-invariant formula} stays the same and the formulae follow from Corollary \ref{c:unitary triangulation} as before.
\end{proof}

\section{Appendix: Plethysm for $\GL(2)$}\label{sect:plethysm}
Let $V$ denote the standard two-dimensional representation of $\GL(2)$ (or $\SL(2)$) (via matrix multiplication). In this section, we study the decomposition
\[\End\Sym^n V\cong\Sym^nV\otimes(\Sym^nV)^\vee\cong\bigoplus_{k=0}^n\Sym^{2k}V\otimes{\det}^{-k}\]
as representations of $\GL(2)$, or alternatively, of $\End\Sym^n V\cong\bigoplus_{k=0}^n\Sym^{2k}V$
as representations of $\SL(2)$.

In general, if $\mathfrak{g}$ is a Lie algebra acting on a finite dimensional vector space $W$, then $\mathfrak{g}$ acts on $W^\vee$ by $(X f)(w) = f(-Xw)$. Moreover, $\mathfrak{g}$ acts on endomorphisms $f:W\to W$ by $(Xf)(w) = X(f(w))+f(-Xw)$ and on $W\otimes W^\vee$ by $X(u\otimes w^\vee) = (X u)\otimes w^\vee+u\otimes (X w^\vee)$. We deduce that there exists a $\mathfrak{g}$-equivariant isomorphism $W\otimes W^\vee\cong \End(W)$ sending $u\otimes w^\vee$ to the endomorphism $x\mapsto w^\vee(x) u$. 

Let $L=\begin{pmatrix} 0&0\\1&0\end{pmatrix}$ and $R=\begin{pmatrix} 0&1\\ 0&0\end{pmatrix}$ be lowering and raising matrices in $\sl(2)$. We choose the basis $(e_1,e_2)$ of $V$ such that $L e_1=e_2, Le_2=0$ and $Re_1=0, Re_2=e_1$. For $m\geq 0$, we denote $V_m = \Sym^mV$ the representation of $\sl(2)$ of highest weight $m$. Define $g_{m,0},\ldots, g_{m,m}$ as the
basis of $V_m$, thought of as a subset of $V^{\otimes m}$, by
\[
g_{m,i}=\binom{m}{i}^{-1}\sum_{\underline{j}}e_{j_1}\otimes\cdots\otimes e_{j_m}:=e_1^{m-1}e_2^i
\]
where the sum is over $m$-tuples $\underline{j}=(j_1,\dots,j_m)\in\{1,2\}^m$ with $j_k=1$ for exactly $m-i$ values of $k$. (When $i\notin \{0,\ldots,m\}$ we simply define
$g_i^{(m)}$ to be 0.) The operators $L$ and $R$ act on $\Sym^mV$ via 
\begin{align*}
Lg_{m,i} &= (m-i)g_{m,i+1}\\
Rg_{m,i} &= ig_{m,i-1}.
\end{align*}

Let $m,n,p\in \mathbb{Z}$ such that $V_p$ appears as a
subrepresentation of $V_m\otimes V_n$, i.e.\ $p\in \{m+n, m+n-2,
\ldots, |m-n|\}$. Denote by $\Xi_{m,n,p}:V_m\otimes V_n\to V_p$
the nontrivial $\mathfrak{sl}(2)$-equivariant projection from
\cite[Lemma 2.7.4]{carter-flath-saito:6j} and
let
\[\psi_{m,n,p}=\left((\sqrt{-1})^{(m+n-p)/2}\frac{m!n!}{((m+n-p)/2)!}\right)\Xi_{m,n,p}\]
which will again be $\mathfrak{sl}(2)$-equivariant. Denote by
$C_{m,n,p}^{u,v,w}$ be the ``inverse Clebsch--Gordan'' coefficients such that
\[\psi_{m,n,p}(g_{m,u}\otimes g_{n,v})=\sum_{w=0}^p C_{m,n,p}^{u,v,w} g_{p,w}\]

\begin{lemma}\label{l:cg}
The coefficients $C_{m,n,p}^{u,v,w}$ satisfy the recurrence
relations
\begin{align*}
(p-w)C_{m,n,p}^{u,v,w}&=(m-u)C_{m,n,p}^{u+1,v,w+1}+(n-v)C_{m,n,p}^{u,v+1,w+1}\\
wC_{m,n,p}^{u,v,w}&=uC_{m,n,p}^{u-1,v,w-1}+vC_{m,n,p}^{u,v-1,w-1}
\end{align*}
and the initial values in the case when $u+v=(m+n-p)/2$ are given by 
\[C_{m,n,p}^{u,v,0}=(-1)^{u}(m-u)!(n-v)!\]
and $C_{m,n,p}^{u,v,w}=0$ for $w>0$. 
\end{lemma}
\begin{proof}
Since the map $\psi_{m,n,p}$ is $\mathfrak{sl}(2)$-equivariant
we get
\begin{align*}
L\psi_{m,n,p}(g_{m,u}\otimes
g_{n,v})&=\psi_{m,n,p}((Lg_{m,u})\otimes g_{n,v}+g_{m,u}\otimes
(Lg_{n,v}))\\
&=(m-u)\psi_{m,n,p}(g_{m,u+1}\otimes g_{n,v})+(n-v)\psi_{m,n,p}(g_{m,u}\otimes
g_{n,v+1})\\
&=(m-u)\sum_w C_{m,n,p}^{u+1,v,w}g_{p,w}+(n-v)\sum_w C_{m,n,p}^{u,v+1,w}g_{p,w}.
\end{align*}
At the same time this is
\begin{align*}
L\psi_{m,n,p}(g_{m,u}\otimes
g_{n,v})&=\sum_w C_{m,n,p}^{u,v,w} Lg_{p,w}\\
&=\sum_w C_{m,n,p}^{u,v,w} (p-w)g_{p,w+1},
\end{align*}
which gives, after identifying the coefficients of $g_{p,w+1}$, the recurrence 
\[(p-w)C_{m,n,p}^{u,v,w}=(m-u)C_{m,n,p}^{u+1,v,w+1}+(n-v)C_{m,n,p}^{u,v+1,w+1}\]
The second recurrence formula is obtained analogously applying the operator $R$ to the definition of the coefficients $C_{m,n,p}^{u,v,w}$.

In \cite[Lemma
2.7.4]{carter-flath-saito:6j} $g_{m,u}$ is
denoted $e_{m/2, m/2-u}$ and the content of the lemma is that
\[\Xi_{m,n,p}(g_{m,u}\otimes g_{n,v}) =
\left((\sqrt{-1})^{(m+n-p)/2}(-1)^{u}\frac{((m+n-p)/2)!(m-u)!(n-v)!}{m!n!}\right)g_{p,0}\]
when $u+v=(m+n-p)/2$. The result follows from the definition of $\psi_{m,n,p}$.
\end{proof}

Lemma \ref{l:cg} gives an explicit map $V_n\otimes V_n\to
V_{2k}$ for $k\leq n$. To make explicit the map $\End(V_n)\to
V_{2k}$, we start with $V_n\cong V_n^\vee$, which is noncanonical as an isomorphism of vector spaces, but can be chosen uniquely (up to scalars) as follows to make the isomorphism $\mathfrak{g}$-equivariant.

Let $(e_1^\vee,e_2^\vee)$ be the basis of $V$ dual to $(e_1,e_2)$, in
which case the dual basis to $g_{n,i}$ is
\[g_{n,i}^\vee=\sum_{\underline{j}}e_{j_1}^\vee\otimes\cdots\otimes e_{j_n}^\vee,
\]
where again the sum is over $\underline{j}$ with $j_k=1$ for exactly $n-i$ values of $k$. The map $\phi$ sending $e_1\mapsto -e_2^\vee$ and $e_2\mapsto e_1^\vee$ is an $\mathfrak{sl}(2)$-equivariant isomorphism $V\cong V^\vee$ and this leads to the $\mathfrak{sl}(2)$-equivariant isomorphism $\phi_n:V_n\to V_n^\vee$ sending $g_{n,i}\mapsto (-1)^{n-i}\binom{n}{i}^{-1}g_{n,n-i}^\vee$. This implies
\begin{align*}
Lg_i^\vee &=-(n+1-i)g_{i-1}^\vee\\
Rg_i^\vee &=-(i+1)g_{i+1}^\vee.
\end{align*}

We remark that, as an $\sl(2)$-representation, $V_m$ has weights
$\{m,m-2,\ldots,-m\}$ and $L$ maps the weight $w$ eigenspace to
the weight $w-2$ eigenspace, while $R$ goes in the other
direction. Moreover, the vector $g_{n,i}\otimes g_{n,j}^\vee\in
V_n\otimes V_n^\vee$ has weight $2(j-i)$ and this implies that
\[v_{2k}=\sum_{i=0}^{n-k} \binom{k+i}{i}g_{n,i}\otimes
g_{n,k+i}^\vee\]
has (highest) weight $2k$. Computationally, this
vector suffices to make explicit the projection $\End(V_n)\to V_{2k}$. Indeed, $V_{2k}$ has basis
$\{(2k-i)!L^iv_{2k}|i=0,\ldots, 2k\}$ and this basis is
proportional to the $(g_{2k,0},\ldots,g_{2k,2k})$. Thus the
projection $\End(V_n)\cong V_n\otimes V_n^\vee\to V_{2k}$ can be
computed by finding the projection of $g_{n,i}\otimes
g_{n,j}^\vee$ to $V_{2k}$ in terms of the basis
$\{(2k-i)!L^iv_{2k}|i=0,\ldots, 2k\}$, which amounts to a matrix
inversion.

However, we will obtain a closed expression for the projection
map using the inverse Clebsch--Gordan coefficients from Lemma
\ref{l:cg}. The endomorphism $g_{n,i}\otimes g_{n,j}^\vee\in
V_n\otimes V_n^\vee\cong\End(V_n)$ projects to $V_{2k}$ and the
composite map is
\begin{align*}
\psi_{n,n,2k}\circ (1\otimes\phi_n^{-1})(g_{n,i}\otimes
g_{n,j}^\vee)&=(-1)^j\binom{n}{j}\psi_{n,n,2k}(g_{n,i}\otimes
g_{n,n-j})\\
&=(-1)^j\binom{n}{j}\sum_{w=0}^{2k}C_{n,n,2k}^{i,n-j,w}g_{2k,w}
\end{align*}

We arrive at the main result of this section:
\begin{proposition}\label{p:End to Sym^n}Suppose the representation $V$ has basis $(e_1,e_2)$ and $V_n=\Sym^nV$ has basis $(g_{n,0},\dots,g_{n,n})$. Suppose $T\in\End(\Sym^nV)$ has an upper
triangular matrix with $(a_0, \ldots, a_{n})$ on the diagonal with respect to this basis. Then the projection of $T$ to $V_{2k}$ is
\[\begin{pmatrix}*&\cdots &*&\displaystyle \sum_{i=0}^{n}B_{n,k,i}a_i&0&\cdots&0 \end{pmatrix}\]
with respect to the basis $(g_{2k,i})$ of $V_{2k}$ where
\begin{align*}
B_{n,k,i}=\sum_{a+b=k}(-1)^{a}\binom{n}{i}\binom{i}{a}\binom{n-i}{b}(n-i+a)!(i+b)!
\end{align*}
Here, the explicit coordinate is the middle one, i.e.\ the coefficient of $g_{2k,k}$, and we use the usual convention that $\binom{x}{y}=0$ if $y<0$ or $y>x$. 
\end{proposition}
\begin{proof}
That $T$ is upper triangular implies that it is a linear
combination of the form $\displaystyle \sum_{i\leq j}
\alpha_{i,j} g_{n,i}\otimes g_{n,j}^\vee$ (here
$a_i=\alpha_{i,i}$). But $g_{n,i}\otimes g_{n,j}^\vee$ has
weight $2(j-i)\geq 0$ and so $T\subset \bigoplus_{w\geq 0} (\End
V_n)_w$. Therefore its projection to $V_{2k}$ belongs to
$\bigoplus_{w\geq 0}(V_{2k})_w$, which is spanned by $g_{2k,u}$ for
$u\leq k$. This implies that the coefficients of $g_{2k,u}$ for
$u>k$ are $0$. Note that the projection $\End V_n\to (V_{2k})_0$
to the weight 0 eigenspace factors through $(\End V_n)_0\to
(V_{2k})_0$ and so the projection of $T$ to $(V_{2k})_0$ only
depends on the image $\sum a_i g_i\otimes g_i^\vee$ of $T$ in
$(\End\Sym^nV)_0$.

The coefficient $B_{n,k,i}$ is then the coefficient of
$g_{2k,k}$ in the projection to $V_{2k}$ of $g_{n,i}\otimes
g_{n,i}^\vee$. By the discussion above this is
\[B_{n,k,i}=(-1)^i\binom{n}{i}C_{n,n,2k}^{i,n-i,k}\]
Lemma \ref{l:cg} implies inductively that for $m\leq k$:
\[C_{n,n,2k}^{i,n-i,k}=\sum_{a+b=m}\binom{i}{a}\binom{n-i}{b}\binom{k}{m}^{-1}C_{n,n,2k}^{i-a,n-i-b,k-m}\]
and so
\[C_{n,n,2k}^{i,n-i,k}=\sum_{a+b=k}(-1)^{i-a}\binom{i}{a}\binom{n-i}{b}(n-i+a)!(i+b)!\]
and the formula follows.
\end{proof}

\begin{remark}\label{r:Bnki}
We end with a computation of the special values of $B_{n,k,i}$ which are involved in our formulae for $\mathcal{L}$-invariants.
In the main formula for $B_{n,k,i}$ in Proposition \ref{p:End to Sym^n}, the indices $a$ and $b$ satisfy $a+b=k$, $a\leq i$ and
$b\leq n-i$. When $k=n$ the only posibility is $(a,b)=(i,n-i)$, when $k=n-1$ the two possibilities are $a\in \{i-1,i\}$, and when $k=n-2$ the three possibilities are $a\in \{i-2,i-1,i\}$. Thus
\begin{align*}
B_{n,n,i}&=(-1)^i(n!)^2\binom{n}{i},\\
B_{n,n-1,i}&=\binom{n}{i}\left((-1)^{i-1}\binom{i}{i-1}\binom{n-i}{n-i}(n-1)!n!+(-1)^i\binom{i}{i}\binom{n-i}{n-1-i}n!(n-1)!\right)\\
&=(-1)^in!(n-1)!\binom{n}{i}(n-2i),\\
B_{n,n-2,i}&=(-1)^i\binom{n}{i}\left(\binom{i}{2}(n-2)!n!-\binom{i}{1}\binom{n-i}{1}((n-1)!)^2+\binom{n-i}{2}n!(n-2)!\right)\\
&=(-1)^i\binom{n}{i}(n-2)!(n-1)!(n^3-(4i+1)n^2+(4i^2+2i)n-2i^2)
\end{align*}
\end{remark}

\paragraph{{\bf Acknowledgements}}
We are grateful to Jo\"el Bella\"iche, Frank Calegari, Mladen Dimitrov, Matthew Emerton, Piper Harron, Ruochuan Liu, Dinakar Ramakrishnan, Sug Woo Shin, Claus Sorensen, and Jacques Tilouine. We also thank Paul Terwilliger for pointing us to the reference \cite{carter-flath-saito:6j} that lead us to the closed-form expression for the $B_{n,k,i}$.

\bibliographystyle{amsalpha}
\bibliography{biblio}

\end{document}